\documentclass[11pt]{article}

\usepackage{array, amsmath, amsfonts, amsthm, amssymb, subcaption, mathrsfs, nicefrac, mathtools, cases}
\usepackage{authblk, color, soul, bm, graphicx, empheq, enumitem, bbold, multirow, booktabs, xspace, accents}
\usepackage[dvipsnames]{xcolor}
\usepackage[margin = 1.5cm, font = small, labelfont = bf, labelsep = period]{caption}
\usepackage[utf8]{inputenc}
\usepackage[T1]{fontenc}
\usepackage[round]{natbib}
\usepackage[ruled]{algorithm2e}
\usepackage{comment, eurosym}
\PassOptionsToPackage{hyphens}{url}\usepackage{hyperref}
\usepackage{tikz}

\usepackage{todonotes}
\usepackage{cleveref}
\Crefname{Th}{Theorem}{Theorems}

\usepackage{fullpage}[2cm]
\theoremstyle{plain}
\newtheorem{Th}{Theorem}
\newtheorem{Prop}{Proposition}

\theoremstyle{definition}
\newtheorem{Ex}{Example}
\newtheorem{Rmk}{Remark}
\newtheorem{Ass}{Assumption}

\newtheorem{lem}{Lemma}

\newcommand{\eg}{\textit{e.g.}}
\newcommand{\ie}{\textit{i.e.}}

\newcommand{\set}[1]{\mathcal{#1}}
\newcommand{\subj}{\rm{s.t.}}

\newcommand{\ubar}[1]{\underaccent{\bar}{#1}}

\newcommand{\D}{\mathcal{D}}
\newcommand{\K}{\mathcal{K}}
\newcommand{\N}{\mathcal{N}}

\newcommand{\T}{\mathcal{T}}
\newcommand{\V}{\mathcal{V}}
\newcommand{\R}{\mathbb R}

\DeclarePairedDelimiter{\ceil}{\lceil}{\rceil}

\DeclareMathOperator{\E}{\mathbb{E}}

\date{}

\begin{document}

\begin{center}
\LARGE
    Reliable Frequency Regulation through Vehicle-to-Grid:\\ Encoding Legislation with Robust Constraints
\end{center}

\begin{center}
Dirk Lauinger\textsuperscript{1}, Fran\c{c}ois Vuille\textsuperscript{2}, Daniel Kuhn\textsuperscript{3}

\footnotesize
    \textsuperscript{1}Massachusetts Institute of Technology,  \href{mailto:lauinger@mit.edu}{lauinger@mit.edu} 
    
    \textsuperscript{2}Etat de Vaud,  \href{mailto:francois.vuille@vaud.ch}{francois.vuille@vaud.ch} 
    
    \textsuperscript{3}Ecole polytechnique fédérale de Lausanne,  \href{mailto:daniel.kuhn@epfl.ch}{daniel.kuhn@epfl.ch}
\end{center}

\noindent \textbf{Problem definition:} Vehicle-to-grid increases the low utilization rate of privately owned electric vehicles by making their batteries available to electricity grids. We formulate a robust optimization problem that maximizes a vehicle owner's expected profit from selling primary frequency regulation to the grid and guarantees that market commitments are met at all times for all frequency deviation trajectories in a functional uncertainty set that encodes applicable legislation. Faithfully modeling the energy conversion losses during battery charging and discharging renders this optimization problem non-convex. \textbf{Methodology/results:} By exploiting a total unimodularity property of the uncertainty set and an exact linear decision rule reformulation, we prove that this non-convex robust optimization problem with functional uncertainties is equivalent to a tractable linear program. Through extensive numerical experiments using real-world data, we quantify the economic value of vehicle-to-grid and elucidate the financial incentives of vehicle owners, aggregators, equipment manufacturers, and regulators. \textbf{Managerial implications:} We find that the prevailing penalties for non-delivery of promised regulation power are too low to incentivize vehicle owners to honor the delivery guarantees given to grid operators.
		

\section{Introduction} 

Replacing internal combustion engine vehicles with electric vehicles reduces urban air pollution and mitigates climate change if electricity is generated from renewable sources~\citep{DS94}. In general, privately owned vehicles are a vastly underutilized resource. Vehicle usage data collected by the US~\cite{NHTS17} shows that on an average day over $90$\% of all privately owned vehicles are parked at any one time---even during peak rush hour. Since electricity grids require storage capacity to integrate increasing amounts of intermittent wind and solar power, electric vehicle owners could capitalize on their batteries by offering storage to the electricity grid when their vehicles are parked. \cite{WK97} term this idea \textit{vehicle-to-grid}.

R\'eseau de transport d'\'electricit\'e (RTE), Europe's largest transmission system operator, expects to need an additional flexible generation and electricity storage capacity of $10$GW to $20$GW by 2035. This corresponds to $7.5$\% to $15$\% of the total French electricity generation capacity in 2017~(RTE~\citeyear{RTE17b, RTE17c}). If electric vehicles were to provide some of this flexibility, then the vehicles and the electricity grid could share the costs of electric vehicle batteries. \cite{WK05a} and \cite{LN19} have identified \textit{primary frequency regulation}, also known as
frequency containment reserves, as one of the most profitable flexibility services for vehicle-to-grid. Electric vehicles that provide this service must maintain a continuous power flow to the vehicle battery that is proportional to the deviation of the instantaneous grid frequency from its nominal value (\eg, 50Hz in Europe). As primary frequency regulation is the first flexibility service used to stabilize the electricity network after disturbances~\citep{YR07a}, its provision must be highly reliable. However, RTE~(\citeyear{RTE17b}) questions the reliability of vehicle-to-grid. The~\cite{EU17} addressed this concern by defining a minimum level of reliability that electric vehicles and other providers of frequency regulation must guarantee. Specifically, it demands that providers must be able to deliver regulation power for all frequency deviation trajectories with certain characteristics.

Adopting the perspective of a vehicle owner, we formulate an optimization model to determine the bidding strategy on the regulation market that maximizes the expected profit from selling primary frequency regulation to the transmission system operator under the reliability constraints imposed by the European Commission. These constraints must hold {\em robustly} for all frequency deviation trajectories in an uncertainty set consistent with applicable legislation. As these trajectories constitute continuous-time functions, we are confronted with a robust optimization problem with functional uncertainties. Moreover, the impossibility of simultaneously charging and discharging the battery---which amounts to dissipating energy through conversion losses and could be profitable when the battery is full and there is a reward for down-regulation (see, {\em e.g.}, \cite[p.~84]{JAT15})---renders the optimization problem non-convex. The main theoretical contribution of this paper is to show that the resulting non-convex robust optimization problem with functional uncertainties is equivalent to a tractable linear program. Specifically, this paper makes the following methodological contributions to robust optimization (see \cite{AB09} for a textbook introduction).
\begin{itemize}
\item We introduce new {\em uncertainty sets in function spaces} that capture those frequency deviation trajectories for which regulation providers must be able to deliver all promised regulation power. These uncertainty sets are reminiscent of the {\em budget uncertainty sets} by \cite{DB04} in finite-dimensional spaces, and their construction is inspired by EU legislation.
\item By leveraging a {\em total unimodularity property} of the proposed uncertainty sets and an {\em exact linear decision rule reformulation}, we prove that the worst-case frequency deviation scenarios in all (convex or non-convex) robust constraints of the vehicle owner's optimization problem can be found by solving continuous linear programs, which can be viewed as variants of the so-called {\em separated continuous linear programs} introduced by \cite{EA83}.
\item By demonstrating that all these continuous linear programs are solved by piecewise constant frequency deviation trajectories, we show that the robust optimization problem with functional uncertainties is equivalent to a {\em robust optimization problem in discrete time}. In doing so, we use more direct proof techniques than \cite{MP95}, who derived sufficient conditions under which the solutions of separated continuous linear programs are piecewise constant. 
\item The robust optimization problem obtained by time discretization is still non-convex. Using the structural properties of its (discretized) uncertainty sets and of its objective and constraint functions, however, we can prove that it is equivalent to a {\em linear} robust optimization problem that can be reformulated as a {\em tractable linear program} via standard techniques.
\end{itemize}

To our best knowledge, robust optimization models with uncertainty sets embedded in function spaces have so far only been considered in the context of robust control, where the primary goal is to develop algorithms for evaluating conservative approximations \citep{BH11}, and in the context of robust continuous linear programming, where the primary goal is to reduce robust to {\em non}-robust continuous linear programs, which can be addressed with existing algorithms \citep{ghate20}. In contrast, we study here a non-convex robust optimization problem with functional uncertainties that admits a lossless time discretization and can be reformulated {\em exactly} as a tractable linear program. Remarkably, the state-of-the-art methods for solving the deterministic counterparts of this robust optimization problem are based on methods from mixed-integer linear programming. To our best knowledge, we thus describe the first class of practically relevant mixed-integer linear programs that simplify to standard linear programs through robustification.

As the emerging linear programs are amenable to efficient numerical solution, we are able to perform extensive numerical experiments based on real-world data pertaining to the French electricity system. We define the~\emph{value of vehicle-to-grid} as the profit from selling primary frequency regulation relative to a baseline scenario in which the vehicle owner does not offer grid services. As our optimization model faithfully captures effective legislation, it enables us to quantify the true value of vehicle-to-grid. This capability is relevant for understanding the economic incentives of different stakeholders such as vehicle owners, aggregators, equipment manufacturers, and regulators. The model developed in this paper enables us to assess how the value of vehicle-to-grid depends on the penalties for non-delivery of promised regulation power, the size of the uncertainty set, and the vehicle's battery and charger. We thus contribute to the growing operations management literature on vehicle-to-grid~\citep{GB17, HYM20, YZ21, WQ21}. The main insights drawn from our computational experiments can be summarized as follows.
\begin{itemize}
	\item Based on 2016--2019 data, we show that the value of vehicle-to-grid attainable with a bidding strategy that is {\em guaranteed} to satisfy all reliability requirements is around $100$\EUR per year and vehicle. Earlier studies based on anticipative bidding strategies have estimated this value to be four times higher~\citep{PC15, OB19}.
	\item We find a similar value of vehicle-to-grid as \cite{PC15} and \cite{OB19} if the vehicle owner risks financial penalties for ignoring the legal reliability requirements. This suggests that \emph{current penalties are too low} to incentivize vehicle owners to respect the law.
	\item We show that the value of vehicle-to-grid saturates at daily plug times above 15 hours. Thus, maximal profits from frequency regulation can be reaped even if the vehicle is disconnected from the grid up to 9 hours per day. This means that vehicle owners still enjoy considerable flexibility as to when to drive, which could help to promote the adoption of vehicle-to-grid. 
    \item We formulate an optimization problem for a vehicle aggregator that allows for asymmetric bids of individual vehicles and show that this problem also admits an exact linear programming reformulation. In practical terms, we find that allowing for asymmetric bids can increase the profits from frequency regulation by up to 40\% for a fleet of electric vehicles with bidirectional and unidirectional chargers.
\end{itemize}

Beyond vehicle-to-grid, this paper contributes to the literature on the optimal usage of energy storage assets. The value of a storage asset is usually identified with the profit that can be generated through arbitrage by trading the stored commodity on spot or forward markets.
Unlike traditional centralized storage assets, decentralized storage assets such as electric vehicles are usually connected to distribution rather than transmission grids. This means that they face retail and not wholesale electricity prices. While wholesale prices are determined by market mechanisms and thus stochastic, retail prices are often regulated and thus deterministic. Another major difference is that it may take several days to fully charge or discharge centralized storage assets such as hydropower plants, whereas the batteries of electric vehicles can be fully charged and discharged in just a few hours. A daily planning horizon is therefore sufficient for optimizing their usage. In addition, typical vehicle owners can anticipate their driving needs at most one day in advance. One can thus solve the storage management problem in a receding horizon fashion. 
	
The state-of-charge of a vehicle battery depends non-linearly on the power in- and outflows, which leads to non-convex optimization models. If the battery is merely used for arbitrage and market prices are non-negative, then these optimization models admit exact convex relaxations. Conversely, if the battery is used for frequency regulation or if market prices can fall below zero, then a non-convex constraint is needed to prevent the models from dissipating energy by simultaneously charging and discharging the battery~\citep{YZ16}. If energy conversion losses are negligible and the battery state-of-charge is thus linear in the power flows, then one can model the provision of frequency regulation through adjustable uncertainty sets. Such an approach has been proposed by \cite{XZ17} for frequency regulation with building appliances. A stochastic dynamic programming scheme for optimizing the charging and discharging policy of an electric vehicle with linear battery dynamics is proposed by \cite{JD14}. If energy conversion losses are significant, however, one may still approximate the state-of-charge by a linear decision rule of the uncertain frequency deviations~\citep{JW13}. \cite{ES12} study a similar model under the assumption of perfect foresight.

In practice, several hundreds or thousands of electric vehicles must be aggregated to be able to bid enough reserve power to qualify for participation in the frequency regulation market. \cite{CG09b}, \cite{SH10}, \cite{GW18} and \cite{YZ21} develop frameworks for controlling the batteries of aggregated vehicles, while the design of contracts between aggregators and vehicle owners is examined by \cite{SH11} and \cite{GB17}. The policy implications for the market entry of electric vehicle aggregators are investigated by \cite{OB18b}. Yet the study of vehicle-to-grid schemes for individual vehicles remains relevant because they constitute important building blocks for aggregation schemes and because they still pose many challenges---especially when it comes to faithfully modeling all major sources of uncertainty.

The model developed in this paper is most closely related to the discrete-time robust optimization models by \cite{EY17} and \cite{EN19}, which capture the uncertainty of the frequency deviations through simplicial uncertainty sets that cover all empirical frequency deviation scenarios. However, these uncertainty sets may fail to include unseen future frequency deviation scenarios and are inconsistent with applicable EU legislation. While \cite{EY17} disregard energy conversion losses, \cite{EN19} account for them heuristically and test the resulting charging and discharging policies experimentally on a real battery. Heuristics are also common in pilot projects that demonstrate the use of vehicle-to-grid for frequency regulation \citep{SV13, SV20}.

The model proposed in this paper relies on three simplifying assumptions that we justify below.

Our first key assumption is that the provision of frequency regulation has no negative impact on battery lifetime---even though the fear of battery degradation has been identified as a major obstacle to the widespread adoption of vehicle-to-grid \citep{EEVC17}. To justify this assumption, we point out that the impact of vehicle-to-grid on battery longevity is not yet well understood. In fact, \cite{MD17} claim that such degradation is severe, while \cite{KU17} claim that vehicle-to-grid may actually extend battery lifetime. In \citep{KU18}, the authors of these two studies reconcile their contradictory findings by concluding that the impact of vehicle-to-grid depends on the operating conditions of the battery, such as its temperature and variations in its state-of-charge. We further justify our no-degradation assumption by restricting the battery state-of-charge to lie within 20\% and 80\% of the nominal battery capacity. \cite{AT18} suggests these restrictions as a rule of thumb for extending the lifetime of common lithium-ion batteries, and \cite{TS17} adopt similar rules to optimize recharging policies of electric vehicles. Models that account for battery degradation are studied by \cite{GH16} and \cite{PC19}.

Our second key assumption is that vehicle owners can specify time and energy windows for their driving needs one day in advance. This assumption makes sense for commuters who adhere to predictable daily routines, for example.

The third key assumption is that the vehicle owners are price takers who influence neither the market prices nor the grid frequency. This assumption is reasonable because one vehicle may cover at most several kilowatts of the 700~megawatts required for frequency regulation in France. A model of a regulation provider influencing the grid frequency is described by~\cite{PM09}.

The paper proceeds as follows. Section~\ref{sec:Prob_Des} formulates the vehicle owner's decision problem as a non-convex robust program with functional uncertainties. In Sections~\ref{sec:time_dis} and~\ref{sec:LPR} we reformulate this problem as a non-convex robust program with vectorial uncertainties and even as a tractable linear program, respectively. In Appendix~\ref{sec:aggregator}, we extend these results to an aggregator who manages a fleet of electric vehicles. Numerical experiments are discussed in Section~\ref{sec:NumEx}, and policy insights are distilled in Section~\ref{sec:conclusions}. All proofs are relegated to the online supplement.

\paragraph*{Notation.} All random variables are designated by tilde signs. Their realizations are denoted by the same symbols without tildes. Vectors and matrices are denoted by lowercase and uppercase boldface letters, respectively. For any $z \in \mathbb{R}$, we define $[ z ]^+ = \max \{z,0\}$ and $[z]^- = \max\{-z,0\}$ such that $z=[z]^+-[z]^-$. The intersection of a set~$\set{A} \subseteq \R^d $ with~$\R^d_+$ is denoted by~$\set{A}^+$. For any closed intervals $\set{T},\set{U} \subseteq \R$, we define~$\set{L}(\T,\set{U})$ as the space of all Riemann integrable functions~$f: \T \to \set{U}$, and we denote the intersection of a set~$\set{B} \subseteq \set{L}(\T,\R)$ with~$\set{L}(\T,\R_+)$ as~$\set{B}^+$.

\section{Problem Description}\label{sec:Prob_Des}
Consider an electric vehicle whose state at any time~$t$ is characterized by the amount of energy~$y(t)$ stored in its battery and the instantaneous power consumption for driving~$d(t)$. We require that $y(t)$~is never smaller than~$\ubar{y}$ and never larger than~$\bar{y}$. To mitigate battery degradation, we set these limits to~20\% and~80\% of the nominal battery capacity, respectively. The battery interacts with the power grid through a bidirectional charger with charging efficiency~$\eta^+ \in (0,1]$ and discharging efficiency $\eta^- \in (0,1]$, where an efficiency of~$1$ corresponds to a lossless energy~conversion between the grid and the battery. The charger is further characterized by its maximum power consumption~$\bar{y}^+(t)$ from the grid and its maximum power provision to the grid~$\bar{y}^-(t)$. The power the battery can charge or discharge is therefore limited by~$\eta^+\bar{y}^+(t)$ and~$\frac{1}{\eta^-}\bar{y}^-(t)$, respectively. Note that~$\bar{y}^+(t)$ and~$\bar{y}^-(t)$ depend on the charger to which the vehicle is connected at time~$t$. When the vehicle is not connected to any charger, \eg, when it is driving, then both~$\bar{y}^+(t)$ and~$\bar{y}^-(t)$ must vanish. A stationary battery can be modeled by setting $d(t)=0$ and keeping $\bar{y}^+(t)$ and $\bar{y}^-(t)$ constant for all $t$.

In order to charge the battery at time $t$, the vehicle owner may buy power~$x^b(t)$ from the local utility at a known time-varying price~$p^b(t)$ as is the case under dynamic pricing schemes or day/night tariffs. In addition, she may also use the vehicle battery to earn extra revenue by providing primary frequency regulation, which can be viewed as an insurance bought by the transmission system operator (TSO) to balance unforeseen mismatches of electricity demand and supply in real time~\citep{JDG10}. If there is more supply than demand, the frequency of the power grid rises. Conversely, if there is more demand than supply, the frequency falls. A battery owner offering regulation power~$x^r(t)$ at time~$t$ is obliged to change her nominal power consumption~$x^b(t)$ from the grid by~$\delta(t) x^r(t)$, where~$\delta(t)$ quantifies the normalized deviation of the instantaneous grid frequency~$f(t)$ from its nominal value~$f_0$~\citep{RTE09}. Formally, we have 
\begin{equation*}
\delta(t) = \left\{ \begin{array}{ll}
+1 	& \text{if } f(t) > f_0 + \Delta f, \\
\frac{f(t)-f_0}{\Delta f} & \text{if } f_0 - \Delta f \leq f(t) \leq f_0 + \Delta f, \\
-1	& \text{if } f(t) < f_0 - \Delta f,
\end{array}\right.
\end{equation*}
where $\Delta f > 0$ is a threshold beyond which all promised regulation power must be delivered.

The TSO contracts frequency regulation as an insurance over a prescribed planning horizon of length~$T$, \eg, one day. The planning horizon is subdivided into trading intervals~$\mathcal{T}_k=[ (k-1)\Delta t, k \Delta t )$ for all~$k \in \K = \{1,\ldots,K\}$, where $K = \frac{T}{\Delta t} \in \mathbb{N}$. In the French electricity market, for example, the length $\Delta t$ of a trading interval is $30$~minutes.
The TSO requests the vehicle owner to announce the market decisions
$x^b(t)$ and~$x^r(t)$ before the beginning of the planning horizon, \eg, one day ahead at noon~(RTE \citeyear{RTE17}). These decisions need to be piecewise constant over the trading intervals. The TSO compensates the vehicle owner for the frequency regulation~$x^r(t)$ made available at the {\em availability price}~$p^a(t)$ and charges her for the increase~$\delta(t) x^r(t)$ in her power consumption at the {\em delivery price}~$p^d(t)$ as laid out by French market rules. Note that this charge becomes negative (\ie, it becomes a remuneration) if $\delta(t)$ is negative. In summary, the vehicle owner's total cost over the planning horizon~$\mathcal{T} = [0,T]$ amounts to
\begin{equation*}
	\int_0^T p^b(t) x^b(t) - \left( p^a(t) - \delta(t)  p^d(t)\right) x^r(t) \, \mathrm{d}t.
\end{equation*}

The impact of providing frequency regulation on the battery state-of-charge depends on how the vehicle owner adjusts the power consumed from and the power injected into the grid to achieve the desired net power consumption $x^b(t) + \delta(t) x^r(t)$. The most energy-efficient way is to avoid unnecessary energy conversion losses resulting from simultaneously charging and discharging. Sometimes, however, such losses can be attractive, for example if the battery is almost full and receives a request for down-regulation~($ \delta(t) > 0$). \cite{YZ16} show that energy losses can also be attractive when electricity prices are negative. Since common chargers are not able to simultaneously charge and discharge, we forbid this option and set the charging rate to
\begin{subequations}
\begin{equation}
y^+\left(x^b(t),x^r(t),\delta(t)\right) = \left[ x^b(t) + \delta(t) x^r(t) \right]^+ \label{eq:y+}
\end{equation}
and the discharging rate to
\begin{equation}
y^-\left(x^b(t),x^r(t),\delta(t)\right) = \left[ x^b(t) + \delta(t) x^r(t) \right]^-. \label{eq:y-}
\end{equation}

\end{subequations}

\begin{Rmk}
	When operating a vehicle fleet, some vehicles could charge while others discharge, which suggests that the regulation profits achievable with $n$~vehicles may exceed the regulation profit of a single vehicle multiplied by~$n$. In this paper, we focus on the case $n=1$.\hfill $\Box$
\end{Rmk}

The power exchanged with the grid and the power needed for driving determine the battery state-of-charge at any time $t$ via the integral equation
\begin{equation}
\label{eq:Inte}
y\left(x^b,x^r,\delta,y_0,t \right) = y_0 + \int_{0}^{t} \eta^+ y^+\left(x^b(t'),x^r(t'),\delta(t')\right) - \frac{y^-\left(x^b(t'),x^r(t'),\delta(t')\right)}{\eta^-} - d(t') \, \mathrm{d}t',
\end{equation}
where~$y_0$ represents the state-of-charge at time~$0$. 

At the time when the vehicle owner needs to choose and report the market commitments $x^b(t)$ and $x^r(t)$, she has no knowledge of the uncertain future frequency deviations $\delta(t)$ and the delivery prices $p^d(t)$ at time $t \in \set{T}$. In addition, she has no means to predict the battery state-of-charge $y_0$ at the beginning of the planning horizon, which depends on market commitments chosen on the previous day and on the uncertain frequency deviations to be revealed until time $0$. By contrast, the availability prices~$p^a(t)$ for~$t \in \T$ can be assumed to be known at the planning time. In practice, these prices are determined by an auction. As the vehicle owner bids an offer curve expressing~$x^r(t)$ as a function of~$p^a(t)$ for any~$t \in \T$, it is as if the availability prices were known upfront (the bidding process is described at \url{https://www.entsoe.eu/network_codes/eb/fcr/}). Next, we describe the information that is available about the uncertain problem parameters~$\delta$, $p^d$, and~$y_0$.

We first discuss the uncertainty in the frequency deviations, which limits the amount of reserve power that can be sold on the market. Indeed, the vehicle owner must ensure that the battery state-of-charge will never drop below $\ubar{y}$ or exceed $\bar{y}$ when the TSO requests down-regulation ($\delta(t) < 0$) or up-regulation ($\delta(t) > 0$), respectively, for a prescribed set of conceivable frequency deviation scenarios. Otherwise, the vehicle owner may not be able to honor her market commitments, in which case the TSO may charge a penalty or even ban her from the market.

The TSO defines under what conditions regulation providers must be able to deliver the promised regulation power, keeping in mind that extreme frequency deviations are uncommon. Indeed, between 2015 and 2018 the frequency deviation $\delta(t)$ has never attained its theoretical maximum of~$1$ or its theoretical minimum of~$-1$ in the French market. In the following, we thus assume that the vehicle owner needs to guarantee the delivery of regulation power \emph{only} for frequency deviation scenarios within the uncertainty set

\begin{equation*}
	\D = \left\{ \delta \in \mathcal{L} \left( \mathcal{T}, \left[-1,1\right] \right) :
		   \int_{\left[t-\Gamma \right]^+}^{t} \left\vert \delta(t') \right\vert \, \mathrm{d}t' \leq \gamma \quad \forall t \in \mathcal{T}\right\}
\end{equation*} 
parametrized by the duration~$\Gamma \in \R_+$ of a \emph{regulation cycle} and the duration~$\gamma \in \R_+$ of an \emph{activation period}. Throughout this paper, we assume that $0 < \gamma \leq \Gamma \leq T$. By focusing on frequency deviation scenarios in~$\D$, one stipulates that consecutive extreme frequency deviations $\delta(t) \in \{-1,1\}$ can occur at most over one activation period within each regulation cycle. The \emph{activation ratio}~$\gamma/\Gamma$ can thus be interpreted as the percentage of time during which the vehicle owner must be able to deliver all committed reserve power.

\begin{Rmk}\label{Rmk:setD}
	Note that the uncertainty set $\D$ grows with $\gamma$ and shrinks with $\Gamma$. \hfill $\Box$
\end{Rmk}

Besides displaying favorable computational properties, the uncertainty set~$\D$ has conceptual appeal because it formalizes the delivery guarantee rules prescribed by the~\citeauthor{EU17}~(2017, art.~156(10, 13b)). These rules stipulate that the ``\textit{minimum activation period to be ensured by} [frequency regulation] \textit{providers }[is not to be] \textit{greater than $30$~or smaller than $15$~minutes.}'' This guideline prompts us to set~$\gamma = 30$~minutes. The EU further demands that regulation providers ``\textit{shall ensure the recovery of} [their] \textit{energy reservoirs as soon as possible, within $2$~hours after the end of the alert state.}''
This means that for $2$~hours after the end of an alert state in which the regulation provider delivered frequency regulation, she will not have to deliver any additional frequency regulation, even if new alert states occur during this time. The law defines an \emph{``alert state''} as a state of normal operating conditions in which prespecified contingencies may lead to abnormal operating conditions but does not specify its duration. 
We conservatively consider a duration equal to the activation period, which maximizes the size of the uncertainty set. In fact, to be consistent, the worst-case alert state should be at least as long as the minimum activation period. The regulation cycle~$\Gamma$ is then equal in duration to the worst-case alert state plus the $2$~hour break from delivering frequency regulation. As the uncertainty set shrinks with $\Gamma$, considering the alert state to last as long as the activation period does indeed maximize the uncertainty set. Thus, we set~$\Gamma = 2.5$~hours.

In the following, we compare the empirical distribution of the daily variance of~$\delta$ between the years 2017 and 2019 with the maximum variance that can be achieved by any hypothetical frequency deviation scenario $\delta \in \D$ for a planning horizon of one day. By slight abuse of notation, we define the variance of a frequency deviation scenario $\delta$ with respect to zero as $\mathrm{Var}(\delta) = \frac{1}{T} \int_{0}^{T} \delta(t)^2 \, \mathrm{d}t$. This is justified because the TSO protects the system against \emph{unforeseen} demand and supply fluctuations, which means that the frequency deviations should be unbiased and thus vanish on average. Indeed, the empirical frequency deviations have an average of $5.98\cdot 10^{-4}$. Figure~\ref{fig:Daily_StdDev} shows that if $T=1$ day, $\gamma = 30$ minutes, and $\Gamma = 2.5$ hours, then the maximum standard deviation of any~$\delta \in \D$ exceeds the maximum empirical standard deviation by a factor of~$2.5$. Thus, $\D$ contains extreme frequency deviation scenarios with unrealistically high variance.
The optimization model developed below not only involves the conservative uncertainty set $\D$ compatible with the guidelines of the European Commission but also a smaller uncertainty set 
\begin{equation*}
\set{\hat{D}} = \left\{ \delta \in \mathcal{L} \left( \mathcal{T}, \left[-1,1\right] \right) :
\int_{[t - \hat{\Gamma}]^+}^t \left\vert \delta(t') \right\vert \, \mathrm{d}t' \leq \hat{\gamma} ~\forall t \in \T \right\}
\end{equation*}
parametrized by $\hat{\Gamma} \geq \Gamma$ and $\hat{\gamma} \leq \gamma$. This uncertainty set contains only frequency deviation scenarios that are likely to materialize under normal operating conditions. Note that $\hat{\D}$ is obtained from $\D$ by inflating $\Gamma$ to $\hat{\Gamma}$ and shrinking $\gamma$ to $\hat{\gamma}$. By Remark~\ref{Rmk:setD}, we may thus conclude that $\hat{\D}$ is indeed a subset of $\D$. While the pessimistic uncertainty set $\D$ is used to enforce the stringent delivery guarantees imposed by the European Commission, the more optimistic uncertainty set $\hat{\D}$ is used to model a softer reachability guarantee for the terminal state-of-charge. In the numerical experiments we will set $\hat{\Gamma} = T = 1$~day and $\hat{\gamma} = \gamma = 30$~minutes. One can show that the variance of all frequency deviation scenarios in $\hat{\D}$ is therefore bounded above by $\Delta t/T = 1/48$. Empirically, this threshold exceeds the variance of the frequency deviation on $99.2\%$~of all days in the years from 2017 to 2018.

\begin{figure}[!t]
	\centering
	\includegraphics[trim=0cm 0cm 0cm 0cm, clip, width=5.5in]{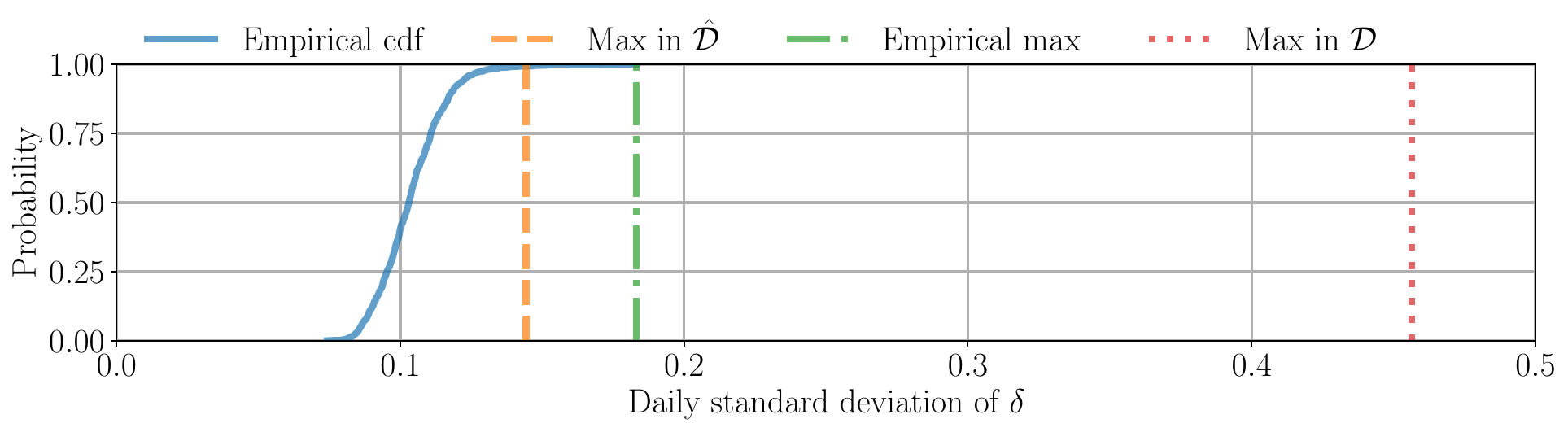}
	\caption{Empirical cumulative distribution function (cdf) of the daily standard deviation of~$\delta$ and the maximum standard deviation of any scenario in~$\hat{\D}$ or in $\D$.}
	\label{fig:Daily_StdDev}
\end{figure}

Next, we discuss the uncertainty in the initial battery state-of-charge $y_0$. Recall that $y_0$ is uncertain at the time when $x^b$ and $x^r$ are chosen because it depends on how much regulation energy must be provided until the beginning of the planning horizon. This quantity depends itself on uncertain frequency deviations that have not yet been revealed. We assume that the vehicle owner constructs two confidence intervals $\set{Y}_0 = [\ubar{y}_0, \bar{y}_0]$ and $\hat{\set{Y}}_0 = [\ubar{\hat{y}}_0, \hat{\bar{y}}_0]$ for $y_0$, either taking into account all frequency deviations under which she must imperatively be able to deliver regulation power or only those frequency deviations that are likely to occur under normal operating conditions.  

The only assumption we make about the uncertainty in the delivery price $p^d$ is that the vehicle owner can reliably estimate the \emph{expected} regulation price $p^r(t) = p^a(t) + \E[\tilde\delta(t) \tilde p^d (t)]$.

We are now ready to formalize the vehicle owner's decision problem for selecting the market decisions $x^b$ and $x^r$. The primary objective is to minimize the expected cost
\begin{align}
c(x^b,x^r) = \E \int_{\T} p^b(t) x^b(t) - \left( p^a(t) + \tilde \delta(t) \tilde p^d(t) \right) x^r(t) \, \mathrm{d}t' = \int_{\T} p^b(t) x^b(t) - p^r(t) x^r(t) \, \mathrm{d}t, \label{eq:obj}
\end{align}
while ensuring that $x^b$ and $x^r$ are robustly feasible across all frequency deviation scenarios $\delta \in \D$ and initial battery states $y_0 \in \set{Y}_0$. Mathematically, the charging rate $y^+(x^b(t),x^r(t),\delta(t))$, the discharging rate $y^-(x^b(t),x^r(t),\delta(t))$, and the battery state-of-charge $y(x^b,x^r,\delta,t,y_0)$ must therefore satisfy the robust constraints
\begin{empheq}[right=\empheqrbrace \text{$\forall t \in \T \text{, } \forall \delta \in \D \text{, } \forall y_0 \in \set{Y}_0 \text{.}$}]{align*}
y^+(x^b(t),x^r(t),\delta(t)) \leq \bar{y}^+(t), \qquad & y(x^b,x^r,\delta, y_0, t) \leq \bar{y},\\
y^-(x^b(t),x^r(t),\delta(t)) \leq \bar{y}^-(t), \qquad& y(x^b,x^r,\delta, y_0, t) \geq \ubar{y}
\end{empheq}

As the vehicle owner continues to use the vehicle for driving and for offering grid services after the end of the planning horizon, the battery should end up in a state that is ``\textit{conducive to satisfactory future operations}'' \citep[p.~1798]{WY85}. Consequently, the vehicle owner aims to steer $y(x^b,x^r,\delta,y_0,T)$ to a desirable state-of-charge~$y$. We assume that the cost-to-go of any~$y\in [\ubar y, \bar y]$ is quantified by a convex and piecewise affine value function~$\varphi(y)=\max_{n \in \set{N}} \{ q_n y + r_n\}$ determined by~$q_n,r_n\in\mathbb R$ for all~$n \in \set{N} = \{1, \ldots, N\}$. As $y_0$ and $\delta$ are uncertain, the terminal state-of-charge is also uncertain. To trade off present versus future costs, it is therefore reasonable to minimize $\varphi(y(x^b,x^r,\delta,y_0,T))$ in view of the worst of all scenarios $\delta \in \hat{\D}$ and $y_0 \in \hat{\set{Y}}_0$. This can be achieved by adding the term $\max_{\delta \in \hat{\D}} \max_{y_0 \in \hat{\set{Y}}_0} \varphi(y(x^b,x^r,\delta,y_0,T))$ to the objective function~\eqref{eq:obj}.

In summary, the vehicle owner's decision problem can be cast as the following robust optimization problem with continuous (functional) uncertain parameters,
\begin{equation}
\tag{R}
\label{pb:Rc}
\begin{array}{>{\displaystyle}c*3{>{\displaystyle}l}}
\min_{x^b, x^r \in \mathcal{X}} & \multicolumn{3}{>{\displaystyle}l}{c(x^b,x^r) + \max_{\delta \in \set{\hat{D}}, \, y_0 \in \set{\hat{Y}}_0} \varphi(y(x^b,x^r,\delta,y_0,T))} \\
\subj & y^+(x^b(t),x^r(t),\delta(t)) &\leq \bar{y}^+(t) & \forall \delta \in \D,~ \forall t \in \T \\
& y^-(x^b(t),x^r(t),\delta(t)) &\leq \bar{y}^-(t) & \forall \delta \in \D,~ \forall t \in \T \\
& y(x^b,x^r,\delta,y_0,t) &\leq \bar{y} & \forall \delta \in \D,~\forall t \in \T,~ \forall y_0 \in \set{Y}_0 \\
& y(x^b,x^r,\delta,y_0,t) &\geq \ubar{y} & \forall \delta \in \D,~ \forall t \in \T,~ \forall y_0 \in \set{Y}_0,
\end{array}
\end{equation}
where~$\mathcal X$ denotes the set of all functions in~$\mathcal{L}(\mathcal{T}, \R_+)$ that are constant on the trading intervals. Using the conservative uncertainty sets $\D$ and $\set{Y}_0$ in the constraints ensures that the delivery guarantee dictated by the European Commission can be fulfilled. Failing to fulfill this guarantee might lead to exclusion from the regulation market. In contrast, there are no drastic consequences of reaching an undesirable state-of-charge at time~$T$. Hence, we use the less conservative uncertainty sets $\hat{\D}$ and~$\hat{\set{Y}}_0$ in the objective function to steer the terminal state-of-charge toward a desirable value under all reasonably likely frequency deviation scenarios. The use of different uncertainty sets in the same model has previously been proposed in robust portfolio insurance problems \citep{SZ11}.

One can show that the function $y(x^b,x^r,\delta,y_0,t)$ is concave in the decision variables~$x^b$ and~$x^r$; see Propostion~\ref{Prop:y} in the online supplement. Upper bounds on this function thus constitute non-convex constraints. This implies that~\eqref{pb:Rc} represents a non-convex robust optimization problem with functional uncertain parameters. In general, such problems are severely intractable.

\begin{Rmk}[Uncertain driving patterns]\label{rmk:driving_patterns}
    Although model~\eqref{pb:Rc} assumes deterministic driving patterns, it readily extends to uncertain driving times and distances. If it is only known that the vehicle will drive at some time within a prespecified interval, then the vehicle owner must not plan on exchanging any electricity with the grid during that interval. Similarly, if it is only known that the vehicle will drive some distance within a certain range, then the vehicle owner must plan with the low end of the range for the constraint on the maximum state-of-charge and with the high end of the range for the constraint on the minimum state-of-charge. The worst-case driving times and distances are thus independent of the vehicle owner's decisions and can be determined \emph{ex-ante}.\hfill $\Box$
\end{Rmk}

\begin{Rmk}[Traditional charging stations]
Model~\eqref{pb:Rc} assumes that the vehicle exclusively connects to the grid through smart charging stations, which enable the vehicle to provide regulation power with or without the option of discharging into the grid. Traditional charging stations may not offer the option of providing frequency regulation. We can model such charging stations by adding the constraint $x^r(t) \leq \bar x^r(t)$ for all $t \in \T$, where $\bar x^r(t) = 0$ if the vehicle is not connected to a smart charger at time~$t$ and $ = \bar y^+(t)$ otherwise, which is always an upper
bound on $x^r(t)$.
\end{Rmk}

\section{Time Discretization}\label{sec:time_dis}

In order to derive a lossless time discretization of the frequency deviation scenarios in problem~\eqref{pb:Rc}, we assume from now on that the power demand for driving and the maximum charge and discharge power of the vehicle charger remain constant over the trading intervals. This assumption is justified because a vehicle that is both driving and parking in the same trading interval cannot offer constant market bids and is therefore unable to participate in the electricity market. Although the power demand for driving may fluctuate wildly, the battery state-of-charge cannot increase while the vehicle is driving, and therefore the power consumption for driving can be averaged over trading intervals without loss of generality. Note that we do \emph{not} assume the frequency deviation scenarios~$\delta$ to remain constant over the trading intervals. In practice $\delta$ may fluctuate on time scales of the order of milliseconds, and averaging out the frequency deviations across a trading interval could result in a dangerous oversimplification of reality. This phenomenon is illustrated in the following example.

\begin{Ex}[Risks of ignoring intra-period fluctuations]
\label{ex:discretization}
As the market decisions~$x^b$ and~$x^r$, the power demand~$d$ and the charging limits~$\bar y^+$ and~$\bar y^-$ are piecewise constant, one might be tempted to replace the frequency deviation signal~$\delta$ with a piecewise constant signal obtained by averaging~$\delta$ over the trading intervals. As we will see, however, averaging~$\delta$ relaxes the battery state-of-charge constraints. Decisions~$x^b$ and~$x^r$ that are {\em in}feasible under the true signal may therefore appear to be feasible under the averaged signal. Hence, replacing the true signal with the averaged signal could make it impossible for the vehicle owner to honor her market commitments. As a simple example, assume that~$x^b(t) = 0$ and $x^r(t)=x^r_1>0$ are constant and that the true frequency deviation signal averages to~0 over the first trading interval~$[0,\Delta t]$, that is, $\frac{1}{\Delta t}\int_0^{\Delta t} \delta(t) \mathrm{d}t = 0$. The left chart of Figure~\ref{fig:avg} visualizes two such signals, which display a small and a high total variation and are denoted by $\delta^{(1)}$ and~$\delta^{(2)}$, respectively. The constant signal equal to their (vanishing) average over~$[0,\Delta t]$ is denoted by~$\delta^{(3)}$. If~$\delta^{(1)}$ reflects reality but is incorrectly replaced with~$\delta^{(3)}$, we are led to believe that the state-of-charge will remain constant at~$y_0$. In reality, however, the battery dissipates the amount~$\Delta\eta \, x^r_1\Delta t/2$ of energy over the first trading interval, where~$\Delta \eta = \frac{1}{\eta^-} - \eta^+ \ge 0$, and the state-of-charge temporarily rises above~$y_0$ by~$\eta^+x^r_1\Delta t/2$. If~$\delta^{(2)}$ reflects reality, on the other hand, then the repeated charging and discharging of the battery still dissipates energy. See the right chart of Figure~\ref{fig:avg} for a visualization. While scenario~$\delta^{(1)}$ is contrived for maximum impact, scenario~$\delta^{(2)}$ rapidly fluctuates around~0 and thus captures a stylized fact that one would expect to see in reality. This example suggests that finding the minimum or the maximum of the state-of-charge over the entire planning horizon and over all signals~$\delta\in \D$ should be non-trivial because intra-period fluctuations {\em do} matter. As a further complication, note that the constraints of the uncertainty set~$\D$ couple the frequency deviations across time. \hfill $\Box$

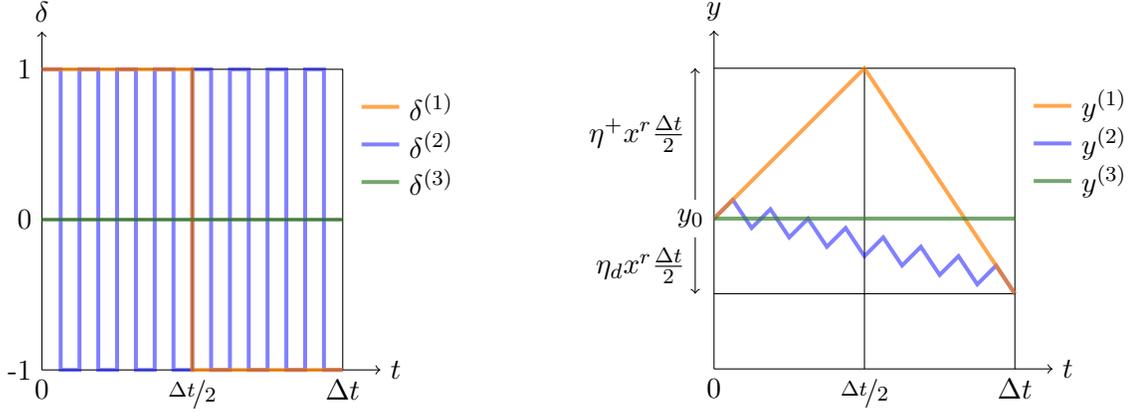
\begin{figure}[t]
	\begin{subfigure}{0.5\textwidth}
		\centering
		\begin{tikzpicture}[scale=0.75, transform shape]
			\draw[step = 2cm] (0, -2) grid (4, 2);
			\draw[->] (0, 2) -- (0, 2.5) node[above] {$\delta$};
			\draw[->] (4, -2) -- (4.5, -2) node[right] {$t$};
			\draw[-] (0,-2) -- (4,-2);
			\node[left] at (0, -2) {-1};
			\node[left] at (0, 0) {0};
			\node[left] at (0, 2) {1};
			\node[below] at (0, -2) {0};
			\node[below] at (2, -2) {$\nicefrac{\Delta t}{2}$};
			\node[below] at (4, -2) {$\Delta t$};
			\draw[color = blue, line width = 1.5pt, opacity = 0.5]
			(0.0, 2) -- (0.25, 2) -- (0.25, -2) -- (0.5, -2) -- (0.5, 2) -- (0.75, 2) -- (0.75, -2) -- (1.0, -2) -- (1.0, 2) -- (1.25, 2) -- (1.25, -2) -- (1.5, -2) -- (1.5, 2) -- (1.75, 2) -- (1.75, -2) -- (2.0, -2) -- (2.0, 2) -- (2.25, 2) -- (2.25, -2) -- (2.5, -2) -- (2.5, 2) -- (2.75, 2) -- (2.75, -2) -- (3.0, -2) -- (3.0, 2) -- (3.25, 2) -- (3.25, -2) -- (3.5, -2) -- (3.5, 2) -- (3.75, 2) -- (3.75, -2) -- (4.0, -2);
			\draw[color = blue, line width = 1.5pt, opacity = 0.5] 	(4.25, 1) -- (4.75, 1);
			\draw[color = orange, line width = 1.5pt, opacity = 0.7]
			(0,2) -- (2,2) -- (2, -2) -- (4, -2)
			(4.25, 1.5) -- (4.75, 1.5);
			\draw[color = OliveGreen, line width = 1.5pt, opacity = 0.7]
			(0,0) -- (4, 0)
			(4.25, 0.5) -- (4.75, 0.5);
			\node[right] at (4.75, 1.5) {$\delta^{(1)}$};
			\node[right] at (4.75, 1) {$\delta^{(2)}$};
			\node[right] at (4.75, 0.5) {$\delta^{(3)}$};
		\end{tikzpicture}
	\end{subfigure}
	\begin{subfigure}{0.5\textwidth}
		\centering
		\begin{tikzpicture}[scale=0.75, transform shape]
			\draw[->] (0, -2) -- (0, 2.5) node[above] {$y$};
			\draw[->] (0, -2) -- (4.5, -2) node[right] {$t$};
			\draw[-] (0,-1) -- (4,-1) (0, 2) -- (4, 2) (2, -2) -- (2, 2) (4, -2) -- (4, 2);
			\draw[->] (-0.25, 0.25) -- (-0.25, 2);
			\draw[->] (-0.25, -0.25) -- (-0.25, -1);
			\node[left] at (-0.25, -0.625) {$\eta_d x^r \frac{\Delta t}{2}$};
			\node[left] at (0, 0) {$y_0$};
			\node[left] at (-0.25, 1.125) {$\eta^+ x^r \frac{\Delta t}{2}$};
			\node[below] at (0, -2) {$0$};
			\node[below] at (2, -2) {$\nicefrac{\Delta t}{2}$};
			\node[below] at (4, -2) {$\Delta t$};
            \draw[color = blue, line width = 1.5pt, opacity = 0.5]
            (0.0, -0.0) -- (0.25,0.25) -- (0.5,-0.125) -- (0.75,0.125) -- (1.0,-0.25) -- (1.25,0.0) -- (1.5,-0.375) -- (1.75,-0.125) -- (2.0,-0.5) -- (2.25,-0.25) -- (2.5,-0.625) -- (2.75,-0.375) -- (3.0,-0.75) -- (3.25,-0.5) -- (3.5,-0.875) -- (3.75,-0.625) -- (4, -1);
			\draw[color = blue, line width = 1.5pt, opacity = 0.5]
			(4.25, 1) -- (4.75, 1);
			\draw[color = orange, line width = 1.5pt, opacity = 0.7]
			(0,0) -- (2, 2) -- (4, -1)
			(4.25, 1.5) -- (4.75, 1.5);
			\draw[color = OliveGreen, line width = 1.5pt, opacity = 0.7]
			(0,0) -- (4, 0)
			(4.25, 0.5) -- (4.75, 0.5);
			\node[right] at (4.75, 1.5) {$y^{(1)}$};
			\node[right] at (4.75, 1) {$y^{(2)}$};
			\node[right] at (4.75, 0.5) {$y^{(3)}$};
		\end{tikzpicture}
	\end{subfigure}
	\caption{Frequency deviation signals (left) and their state-of-charge trajectories (right).}
	\label{fig:avg}
\end{figure}
\end{Ex}

We will now argue that, in spite of Example~\ref{ex:discretization}, $\D$ and~$\hat\D$ can be restricted to contain only piecewise constant frequency deviation signals {\em without} relaxing problem~\eqref{pb:Rc}. 
To formalize the reasoning about piecewise constant functions, we introduce a lifting operator $L:\mathbb{R}^K \to \set{L}(\T,\mathbb{R})$ that maps any vector $\bm{v} \in \R^K$ to a piecewise constant function $L\bm{v}$ with $K$ pieces defined through $(L \bm{v} )(t) = v_k$ if $t \in \T_k$, $k \in \K  = \{1, \ldots, K\}$. We also introduce the adjoint operator $L^\dagger: \set{L}(\T,\mathbb{R}) \to \mathbb{R}^K$ that maps any function $w\in\set{L}(\T,\R)$ to a $K$-dimensional vector $L^\dagger w$ defined through $(L^\dagger w)_k = \frac{1}{\Delta t} \int_{\T_k}w(t) \, \mathrm{d}(t)$ for all $k \in \K$. Note that $L$ and $L^\dagger$ are indeed adjoint to each other because $\int_{\T} (L \bm{v})(t) w(t) \, \mathrm{d}t = \bm{v}^\top L^\dagger(w) $ for all $\bm{v} \in \mathbb{R}^K$ and $w \in \set{L}(\T,\mathbb{R})$. Mathematically, we impose from now on the following assumption.

\begin{Ass}\label{Ass:cst}
The functions $d$, $\bar{y}^+$ and $\bar{y}^-$ are piecewise constant, that is, there exist $\bm{d}, \bar{\bm{y}}^+, \bar{\bm{y}}^- \in \mathbb{R}^K$ such that $d = L\bm{d}$, $\bar{y}^+ = L \bar{\bm{y}}^+$ and $\bar{y}^- = L \bar{\bm{y}}^-$.
\end{Ass}

Next, we introduce a discretized uncertainty set
\begin{equation*}
\D_\K = \left\{ \bm{\delta} \in \left[-1,1\right]^K : \sum\limits_{l=1+ \left[k - \Gamma/\Delta t\right]^+}^{k} \vert \delta_l \vert \leq \frac{\gamma}{\Delta t} ~ \forall k \in \K \right\}
\end{equation*} 
reminiscent of $\D$, where $\Gamma/\Delta t$ and $\gamma/\Delta t$ count the trading intervals within a regulation cycle and an activation period, respectively. Similarly, we define a smaller discretized uncertainty set~$\hat{\D}_\K \subseteq \R^K$ reminiscent of $\hat{\D}$, which is obtained from $\D_{\K}$ by replacing $\Gamma$ with $\hat{\Gamma}$ and $\gamma$ with $\hat{\gamma}$. In the remainder we impose the following divisibility assumption.

\begin{Ass}\label{Ass:div}
The parameters $\Gamma$, $\gamma$, $\hat{\Gamma}$ and $\hat{\gamma}$ are (positive) multiplies of $\Delta t$.
\end{Ass}

Recall that $\Gamma$ and $\gamma$ characterize the delivery guarantees of frequency regulation bids, while $\Delta t$ characterizes the granularity of the market bids. In principle, $\Delta t$ is unrelated to $\gamma$ and $\Gamma$. Assumption~\ref{Ass:div} improves the tractability of model~\ref{pb:Rc} and comes at hardly any loss of generality. If the assumption does not hold but $\Delta t$, $\gamma$, and $\Gamma$ are rational, it can be enforced by reducing $\Delta t$ to the greatest common divisor of $\gamma$, $\Gamma$, and the original $\Delta t$, and by introducing linear coupling constraints that will ensure piecewise constant market decisions over the original trading intervals.

Next, we define the finite-dimensional feasible set $\set{X}_\K = L^\dagger \set{X}$. 
As $\set{X}$ contains only piecewise constant functions, we have $\set{X} = L \set{X}_\K$. We further define the cost function $c_\K(\bm{x^b}, \bm{x^r}) = c(L \bm{x^b},L \bm{x^r})$, which is linear in $\bm{x^b} \in \mathbb{R}^K$ and $\bm{x^r} \in \mathbb{R}^K$. In addition, for any $k \in \K$ we define the function
\begin{align}
	y_{k}\left(\bm{x^b},\bm{x^r},\bm{\delta},y_0\right) & = y\left(L\bm{x^b}, L\bm{x^r}, L \bm{\delta}, y_0, k \Delta t\right) \notag \\ 
	& = y_0 + \Delta t \sum_{l=1}^{k} \eta^+ y^+(x^b_l,x^r_l,\delta_l) - \frac{1}{\eta^-} y^-(x^b_l,x^r_l,\delta_l) - d_l,
\end{align}
which represents the battery state-of-charge at the end of period~$k$ under the assumption that both the market bids and the frequency deviations are piecewise constant.

We are now ready to define the discrete-time counterpart of the robust optimization problem~\eqref{pb:Rc}.
\begin{equation}
\tag{R$_\K$}
\label{pb:R}
\begin{array}{>{\displaystyle}c*3{>{\displaystyle}l}}
\min_{\bm{x^b},\bm{x^r} \in \set{X}_\K} & \multicolumn{3}{>{\displaystyle}l}{c_\K(\bm{x^b},\bm{x^r}) + \max_{\bm{\delta} \in \set{\hat{D}}_\K, y_0 \in \set{\hat{Y}}_0} \varphi(y_K(\bm{x^b},\bm{x^r},\bm{\delta},y_0))} \\
\subj & y^+(x^b_k,x^r_k,\delta_k) &\leq \bar{y}^+_k & \forall \bm{\delta} \in \D_\K, ~\forall k \in \K \\
& y^-(x^b_k,x^r_k,\delta_k) &\leq \bar{y}^-_k & \forall \bm{\delta} \in \D_\K, ~\forall k \in \K \\
& y_{k}(\bm{x^b},\bm{x^r},\bm{\delta},y_0) &\leq \bar{y} & \forall \bm{\delta} \in \D_\K, ~\forall k \in \K \cup \{0\}, ~\forall y_0 \in \set{Y}_0 \\
& y_{k}(\bm{x^b},\bm{x^r},\bm{\delta},y_0) &\geq \ubar{y} & \forall \bm{\delta} \in \D_\K, ~\forall k \in \K \cup \{0\}, ~\forall y_0 \in \set{Y}_0
\end{array}
\end{equation}

Unlike the original problem~\eqref{pb:Rc}, the discrete-time counterpart~\eqref{pb:R} constitutes a standard robust optimization problem that involves only finite-dimensional uncertain parameters. For this reason, there is hope that $\eqref{pb:R}$ is easier to solve than~$\eqref{pb:Rc}$. 

\begin{Th}[Lossless time discretization]\label{th:time}
	The problems~\eqref{pb:Rc} and~\eqref{pb:R} are equivalent.
\end{Th}

The equivalence of~\eqref{pb:Rc} and~\eqref{pb:R} is perhaps surprising in view of Example~\ref{ex:discretization}. It means that the worst-case frequency deviation scenarios are piecewise constant even though intra-period fluctuations matter. Theorem~\ref{th:time} is proved by showing that the four robust constraints in~\eqref{pb:Rc} with functional uncertainties are equivalent to the corresponding robust constraints in~\eqref{pb:R} with vectorial uncertainties and that the worst-case terminal cost functions in~\eqref{pb:Rc} and~\eqref{pb:R} coincide. For example, the equivalence of the first (second) robust constraints in~\eqref{pb:Rc} and~\eqref{pb:R} follows from the observation that, for any fixed~$t\in\mathcal T$, the left hand side of the first (second) constraint in~\eqref{pb:Rc} is maximized by a scenario~$\delta \in\D$ with~$\delta(t)=1$ ($\delta(t)=-1$), which exists by the definition of~$\D$. The last two robust constraints in~\eqref{pb:Rc} are nonlocal as they depend on the entire frequency deviation scenario~$\delta$ and not only on its value at a particular time. Thus, they are significantly more intricate. The robust upper bound on the state-of-charge can be reformulated as an upper bound on $\max_{t\in\mathcal T} \max_{\delta\in \D}y(x^b,x^r,\delta,y_0,t)$. Thanks to the properties of the state-of-charge and the uncertainty set, one can then show that the maximum over~$t\in\mathcal T$ must be attained at~$t=k\Delta t$ for some~$k\in\mathcal K\cup\{0\}$ and that for any such~$t$ the state-of-charge can be expressed as an integral of~$\delta$ against a piecewise constant function. Thus, averaging~$\delta$ across the trading intervals has no impact on the state-of-charge, which in turn allows us to focus on piecewise constant scenarios without restricting generality. The robust lower bound on the state-of-charge in~\eqref{pb:Rc} can be reformulated as a lower bound on $\min_{t\in\mathcal T} \min_{\delta\in \D}y(x^b,x^r,\delta,y_0,t)$, which appears to be intractable because the optimization problem over~$\delta$ minimizes a concave function over a convex feasible set and is therefore non-convex. Classical robust optimization provides no general recipe for handling such constraints even if the uncertain parameters are finite-dimensional, and state-of-the-art research settles for deriving approximations \citep{roos2018approximation}. By exploiting a continuous total unimodularity property of the uncertainty set~$\D$ facilitated by Assumption~\ref{Ass:div}, we first prove that the minimum of~$y(x^b,x^r,\delta,y_0,t)$ over~$\D$ is attained by a frequency deviation trajectory that takes only values in~$\{-1,0\}$. Next, we demonstrate that there exists an affine function of~$\delta$ that matches~$y(x^b,x^r,\delta,y_0,t)$ for all trajectories~$\delta\in\D$ valued in~$\{-1,0\}$ and for all~$t\in\mathcal T$. In the language of robust optimization,  the state-of-charge~$y(x^b,x^r,\delta,y_0,t)$ can be viewed as an analysis variable that adapts to the uncertainty~$\delta$, and the corresponding affine function constitutes a decision rule approximating~$y(x^b,x^r,\delta,y_0,t)$. Decision rule approximations almost invariably introduce approximation errors  \citep[\S~14]{AB09}.
However, the affine decision rule proposed here is error-free because it coincides with~$y(x^b,x^r,\delta,y_0,t)$ for all scenarios~$\delta\in\D$ valued in~$\{-1,0\}$ that may attain the worst case in~$\min_{t\in\mathcal T} \min_{\delta\in \D}y(x^b,x^r,\delta,y_0,t)$. Using this decision rule in an elaborate sensitivity analysis, we can finally prove that the minimum over~$t\in\mathcal T$ must be attained at~$t=k\Delta t$ for some~$k\in\mathcal K\cup\{0\}$ and that for any such~$t$ the state-of-charge can be expressed as an integral of~$\delta$ against a piecewise constant function. Thus, we may focus again on piecewise constant scenarios without restricting generality. The full proof of Theorem~\ref{th:time} can be found in the online supplement.

The new robust optimization techniques developed to prove Theorem~\ref{th:time} are of independent interest as they provide exact tractable reformulations for certain adjustable robust optimization problems with functional or vectorial uncertain parameters, where the embedded optimization problems over the uncertainty realizations are non-convex. We also note that the embedded optimization problems over~$\delta\in\D$ in problem~\eqref{pb:Rc} can be viewed as variants of the so-called {\em separated continuous linear programs} introduced by \cite{EA83}. The proof of Theorem~\ref{th:time} shows that these problems are solved by piecewise constant frequency deviation scenarios that can be computed efficiently, thereby extending the purely existential results by \cite{MP95}. 

Even though the non-convex robust optimization problem~\eqref{pb:Rc} with functional uncertainty admits a lossless time discretization, its discrete-time counterpart~\eqref{pb:R} still constitutes a non-convex robust optimization problem and thus appears to be hard. In the next section, however, we will show that~\eqref{pb:Rc} can be reformulated as a tractable linear program by exploiting its structural properties.

\section{Linear Programming Reformulation}
\label{sec:LPR}
In order to establish the tractability of the non-convex robust optimization problem~\eqref{pb:Rc}, it is useful to reformulate its time discretization~\eqref{pb:R} as the following \emph{linear} robust optimization problem, where all constraint functions are bilinear in the decision variables and the uncertain parameters.
\begin{equation}
	\tag{R$'_\K$}
	\label{pb:LR}
	\begin{array}{>{\displaystyle}c*3{>{\displaystyle}l}}
		\min & c_\K(\bm{x^b}, \bm{x^r}) + z \\
		\subj
		& \multicolumn{3}{*1{>{\displaystyle}l}}{
		\bm{x^b}, \bm{x^r} \in \set{X}_{\K}, ~ \bm m \in \mathbb{R}^K 
		}\\
		& \multicolumn{2}{*1{>{\displaystyle}l}}{
			x^r_k + x^b_k \leq \bar{y}^+_k, 
			\quad
			x^r_k - x^b_k \leq \bar{y}^-_k}
		& \forall k \in \K \\
		& \multicolumn{1}{*1{>{\displaystyle}l}}{
			m_k - \eta^+ x^r_k \geq 0,
			\quad
			m_k - \frac{1}{\eta^-} x^r_k + \Delta \eta \, x^b_k} & \geq 0
		& \forall k \in \K \\
		& \bar{y}_0 + \Delta t \sum_{l=1}^{k} \eta^+ \big( x^b_l + \delta_l x^r_l \big) - d_l & \leq \bar{y} & \forall \bm \delta \in \D_\K^+,~\forall k \in \K \cup \{0\} \\
		& \ubar{y}_0 + \Delta t \sum_{l=1}^{k} \eta^+ x^b_l - m_l \delta_l - d_l & \geq \ubar{y} & \forall \bm \delta \in \D_\K^+,~\forall k \in \K \cup \{0\} \\
		& \hat{\bar{y}}_0 + \Delta t \sum_{k=1}^{K} \eta^+ \big( x^b_k + \delta_k x^r_k \big) - d_k & \leq \frac{z-r_n}{q_n}
		& \forall \bm \delta \in \hat \D^+_\K,~ \forall n \in \mathcal{N}_+ \\
		& \hat{\ubar{y}}_0 + \Delta t \sum_{k=1}^{K} \eta^+ x^b_k - m_k \delta_k - d_k & \geq \frac{z-r_n}{q_n} & \forall \bm \delta \in \hat \D_\K^+,~ \forall n \in \mathcal{N}_-\\
		& z \ge r_n & & \forall n\in\N_0
	\end{array}
\end{equation}
Here, $\Delta \eta = \frac{1}{\eta^-} - \eta^+ \ge 0$ is used as a shorthand for the reduction in the battery state-of-charge resulting from first charging and then discharging one unit of energy as seen from the grid. In addition, we set $\set{N}_+ = \{n \in \set{N}: q_n > 0 \}$, $\set{N}_- = \{n \in \set{N}: q_n < 0 \}$ and $\set{N}_0 = \set{N} \setminus (\N_+\cup\N_-)$.

\begin{Th}[Lossless linearization]\label{th:lr}
	The problems~\eqref{pb:R} and~\eqref{pb:LR} are equivalent.
\end{Th}

The proof of Theorem~\ref{th:lr} critically relies on the exact affine decision rule approximation discovered in the proof of Theorem~\ref{th:time}. Note that the linear robust optimization problem~\eqref{pb:LR} still appears to be difficult because each robust constraint must hold for all frequency deviation scenarios in an uncountable uncertainty set~$\D_{\K}^+$ or~$\hat \D^+_\K$ and therefore corresponds to a continuum of ordinary linear constraints. Fortunately, standard robust optimization theory~\citep{AB04, DB04} allows us to reformulate~\eqref{pb:LR} as the tractable linear program

\begin{equation}
\label{pb:LP}
\tag{R$^{''}_\K$}
\begin{array}{*1{>{\displaystyle}c}*2{>{\displaystyle}l}}
\min & \multicolumn{2}{>{\displaystyle}l}{c_\K(\bm{x^b}, \bm{x^r})  + z} \\ 
\subj
& \multicolumn{2}{>{\displaystyle}l}{
\bm{x^b}, \bm{x^r} \in \set{X}_\K,~z\in\R, ~
\bm m, \bm \lambda^+, \bm \lambda^-, \bm \theta^+, \bm \theta^- \in \R^K_+,~
\bm \Lambda^+, \bm \Lambda^-, \bm \Theta^+, \bm \Theta^- \in \mathbb{R}^{K\times K}_+
} \\
& \multicolumn{1}{>{\displaystyle}l}{
x^r_k + x^b_k \leq \bar{y}^+_k,
\quad
x^r_k - x^b_k
\leq \bar{y}^-_k} & \forall k \in \K \\
& \multicolumn{1}{>{\displaystyle}l}{
m_k \geq \eta^+ x^r_k, 
\quad
m_k \geq \frac{1}{\eta^-}x^r_k - \Delta \eta \, x^b_k }
& \forall k \in \K \\	
& \sum_{l=1}^k \Delta t \left( \eta^+x^b_l + \Lambda^+_{k,l} - d_l \right)  + \gamma \Theta^+_{k,l} \leq \bar{y} - \bar y_0 & \forall k \in \K \cup\{0\}\\
&\sum_{l = 1}^k \Delta t \left( \eta^+x^b_l - \Lambda^-_{k,l} - d_l \right) - \gamma \Theta^-_{k,l} \geq \ubar{y} - \ubar y_0 & \forall k \in \K \cup\{0\}\\
& \sum_{k \in \K} \Delta t \left( \eta^+x^b_k + \lambda^+_{k} - d_k \right) + \hat \gamma \theta^+_{k} \leq \frac{z-r_n}{q_n} - \hat{\bar y}_0 & \forall n \in \mathcal{N}_+ \\
& \sum_{k \in \K} \Delta t \left( \eta^+x^b_k - \lambda^-_{k} - d_k \right) - \hat \gamma \theta^-_{k} \geq \frac{z-r_n}{q_n} - \hat{\ubar y}_0 & \forall n \in \mathcal{N}_- \\
& b_n\leq z & \forall n\in\N_0\\
& \Lambda^+_{k,l} + \sum\limits_{i = l}^{I(k,l)} \Theta^+_{k,i} \geq \eta^+ x^r_l, \quad
\Lambda^-_{k,l} + \sum\limits_{i = l}^{I(k,l)} \Theta^-_{k,i} \geq m_l & \forall k,l \in \K:\; l \leq k \\
& \lambda^+_{k} + \sum\limits_{i = k}^{\hat I(K,k)} \theta^+_{i} \geq \eta^+ x^r_k, \quad
\lambda^-_{k} + \sum\limits_{i = k}^{\hat I(K,k)} \theta^-_{i} \geq m_k & \forall k \in \K,
\end{array}
\end{equation}
where $I(k,l) = \min\{k, l+ \Gamma/\Delta t - 1\}$ and $\hat I(k,l) = \min\{k, l+ \hat \Gamma/\Delta t -1\}$.

\begin{Th}[Linear programming reformulation]\label{th:lp}
	The problems~\eqref{pb:LR} and~\eqref{pb:LP} are equivalent.
\end{Th}

The conversion of the robust optimization problem~\eqref{pb:LR} to the linear program~\eqref{pb:LP} comes at the expense of introducing $4K^2 + 4K$ dual variables. Overall, the linear program~\eqref{pb:LP} involves $4K^2 + 7K + 1$ variables and $K^2 + 9 K + N + 2$~constraints, that is, its size scales quadratically with~$K$. In conjunction, Theorems~1--3 imply that the non-convex robust optimization problem~\eqref{pb:Rc} with continuous uncertain parameters can be reduced without any loss to the tractable linear program~\eqref{pb:LP}, which is amenable to efficient numerical solution with state-of-the-art linear programming solvers.

\begin{Rmk}[Robustification reduces complexity]
    A striking property of the robust optimization model~\eqref{pb:Rc} is that it is much {\em easier} to solve than the underlying deterministic model, which would assume precise knowledge of the frequency deviation scenario~$\delta$. Indeed, the textbook formulation of the deterministic model requires continuous decision variables to represent~$y^+(x^b(t),x^r(t),\delta(t))$ and~$y^-(x^b(t),x^r(t),\delta(t))$ and a binary decision variable to model their complementarity for every~$t~\in~\set T$ \cite[p.~85]{JAT15}. This results in a large-scale mixed-integer linear program even if~$\set T$ is discretized. In contrast, the robust optimization model~\eqref{pb:Rc} is equivalent to the tractable linear program~\eqref{pb:LP}. To our best knowledge, we have thus discovered the first practically interesting class of optimization problems that become dramatically easier through robustification. Since the uncertainty set~$\D$ can be seen as a continuous-time version of a classical budget uncertainty set~\citep{DB04,CB12}, we hypothesize that our approach applies to other robust optimization problem with continuous-time budget uncertainty sets. Budget uncertainty sets are very popular in classical robust optimization, and we suspect that there are similarly many applications for robust optimization with functional uncertainties.
    \hfill $\Box$
\end{Rmk}

\begin{Rmk}[Futility of solving a multistage model] The proposed model~\eqref{pb:Rc} looks only one day ahead and accounts for the future usage of the vehicle only through the cost-to-go function~$\varphi$, which depends on the terminal state-of-charge estimates. The basic model~\eqref{pb:Rc} can readily be extended to a robust dynamic program that looks $H$ days into the future. The cost-to-go functions of this dynamic program can be shown to be convex and piecewise affine. In addition, it is possible to construct tight convex piecewise affine bounds on these cost-to-go functions by solving a series of linear programs. Numerically, we find that the added value of solving a dynamic model that looks $H > 1$ days into the future is negligible, probably because electric vehicles can usually fully recharge overnight. We emphasize, however, that the robust optimization models and techniques developed in this paper may also be useful to optimize the operation of other energy storage devices that are characterized by slower dynamics and therefore necessitate a proper multi-stage approach. \hfill $\Box$
\end{Rmk}

\section{Numerical Experiments}\label{sec:NumEx}
In the following, we first describe how the vehicle owner's decision problem is parametrized from data, and we explain the backtesting procedure that is used to assess the performance of a given bidding strategy. Next, we present numerical results and discuss policy implications. All experiments are run on an Intel i7-6700 CPU with $3.40$GHz clock speed and $64$GB of RAM. All linear programs are solved with GUROBI~$9.1.2$ using its PYTHON interface. In order to ensure the reproducibility of our experiments, we provide links to all data sources and make our code available at \url{www.github.com/lauinger/reliable-frequency-regulation-through-vehicle-to-grid}. 

\subsection{Model Parametrization}\label{sec:Params}
Our experiments are based on availability and delivery prices as well as frequency measurements from the French transmission system. There have been two policy changes in frequency regulation since 2015. While the availability prices were historically kept constant throughout the year, they change on a weekly basis since mid-January 2017 and on a daily basis since July 2019. At this point, the pricing mechanism also changed from a pay-as-bid auction to a clearing price auction. The average availability price over all years from 2015 to 2019 amounts to $0.8$cts/kWh, but the yearly average \emph{decreased} in 2017 and 2018, and \emph{increased} again in 2019 to pre-2017 levels. For all practical purposes we may assume that the expected regulation price $p^r(t) = p^a(t) + \E[\tilde\delta(t) \tilde p^d (t)]$ coincides with the availability price~$p^a(t)$ because the realized regulation price~$p^a(t) + \delta(t) p^d(t)$ oscillates rapidly around the availability price~$p^a(t)$ due to intra-day fluctuations of the frequency-adjusted delivery prices; see Figure~\ref{fig:Reserve price} in Appendix~\ref{Apx:Data}. In fact, $\delta(t) p^d(t)$ empirically averages to~$-2.36 \cdot 10^{-5}$\EUR~over all 10s intervals from 2015 to 2019. We further identify the utility prices~$p^b(t)$ with the residential electricity prices charged by Electricité de France (EDF), the largest European electricity provider. These prices exhibit six different levels corresponding to peak- and off-peak hours on high, medium and low price days. High price days can occur exclusively on work days between November and March, whereas medium price days can occur on all days except Sundays. Low-price days can occur year-round. The peak hours are defined as the hours from 6~am to 10~pm on work days, and all the other hours are designated as off-peak hours. The prices corresponding to each type of day and hour are regulated and published in the official French government bulletin. From 2015 to 2019, these prices have not changed more than three times per year. On each day, RTE announces the next day's price levels by 10:30~am. The average utility price over the years 2015 to 2019 amounts to $14$cts/kWh, exceeding the average availability price by an order of magnitude. 

When simulating the impact of the market decisions on the battery state-of-charge, it is important to track the frequency signal with a high time resolution. In fact, the \cite{EU17} requires regulation providers to adjust the power flow between the battery and the grid every ten seconds to ensure that it closely matches $x^b(t) + \delta(t) x^r(t)$ for all $t \in \T$. We thus use a sampling rate of~$100$mHz when simulating the impact of the market decisions on the battery state-of-charge. 

The vehicle data is summarized in Table~\ref{tab:Params} in Appendix~\ref{Apx:Data}. The chosen parameter values are representative for commercially available midrange vehicle-to-grid-capable electric vehicles such as the 2018~Nissan Leaf. We assume that the vehicle owner reserves the time windows from 7~am to 9~am and from 5~pm to 7~pm on workdays and from 8~am to 8~pm on weekends and public holidays for driving. At all other times, the car is connected to a bidirectional charging station. We also assume that the car travels about 27km per day, which could be easily covered within one hour. However, it makes sense to reserve extended time slots for driving because vehicle owners may not be able to (nor wish to) pinpoint the exact driving times one day in advance.

\subsection{Backtesting Procedure and Baseline Strategy}\label{sec:backtest}

In our experiments, we assess the performance of different bidding strategies over different test datasets covering one of the years between 2015 and 2019. A bidding strategy is any procedure that computes on each day at noon a pair of market decisions $x^b$ and $x^r$ for the following day. We call a strategy \emph{non-anticipative} if it determines the market decisions using only information observed in the past. In addition, we call a strategy \emph{feasible} if it allows the vehicle owner to honor all market commitments for all frequency deviation scenarios within the uncertainty set~$\D$.

To measure the profit generated by a particular strategy over one year of test data, we use the following backtesting procedure. On each day at noon we compute the market decisions for the following day. We then use the actual frequency deviation data between noon and midnight and the market decisions for the current day to calculate the true battery state-of-charge at midnight. Next, we use the frequency deviation data of the following day to calculate the revenue from selling regulation power to the TSO, which is subtracted from the cost of buying electricity for charging the battery. If the strategy is infeasible and the vehicle owner is not able to deliver all promised regulation power even though the realized frequency deviation trajectory falls within the uncertainty set~$\D$, then she pays a penalty. The penalty at time~$t$ is set to~$k_{\mathrm{pen}}\cdot p^a(t)\cdot(x^r(t) - x^r_\mathrm{d}(t))$, where $x^r_\mathrm{d}(t)$ denotes the maximum amount of regulation power that could have been offered without risking an infeasibility, and the penalty factor~$k_{\mathrm{pen}}$ ranges from~$3$ to~$10$, \eg, RTE~(\citeyear{RTE17}) sets $k_{\mathrm{pen}} = 5$. Repeated offenses may even lead to market exclusion. For simplicity, in each experiment we either assume that the vehicle owner pays a penalty corresponding to a fixed value of $k_{\mathrm{pen}}$ for every offense or is excluded from the regulation market directly upon the first offense. 
If the battery is depleted during a trip, any missing energy needed for driving is acquired at a high price~$p^y$ from a public fast charging station. We assume that~$p^y$ accounts for the price of energy as well as for the opportunity cost of the time lost in driving to the charging station and waiting to be serviced. In our experiments we set $p^y$ either to 0.75\EUR/kWh 
(which corresponds to typical energy prices offered by the European fast charging network Ionity), to~7.5\EUR/kWh or to~75\EUR/kWh. The procedure described above is repeated on each day, and the resulting daily profits are accumulated over the entire test dataset.

Our baseline strategy is to determine the next day's market decisions by solving the robust optimization problem~\eqref{pb:Rc} with terminal cost function~$\varphi(y) = p^\star \vert y - y^\star \vert$, where the calibration of~$y^\star$ and~$p^\star$ is described below. Thus, \eqref{pb:Rc} is equivalent to an instance of the linear program~\eqref{pb:LP}. This problem is updated on each day because the driving pattern~$d$ as well as the market prices~$p^b$ and~$p^r$ change, and because the uncertainty sets $\set{Y}_0$ and $\hat{\set{Y}}_0$ for the state-of-charge at midnight depend on the state-of-charge at noon and on the market commitments between noon and midnight that were chosen one day earlier. The baseline strategy is feasible thanks to the robust constraints in~\eqref{pb:Rc}, which ensure that regulation power can be provided for all frequency deviation scenarios in~$\D$.

The parameters $p^\star$ and $y^\star$ are kept constant throughout each backtest. Specifically, we set $p^\star = \frac{3+k}{40}$\EUR/kWh for some $k = 1,\ldots,9$ and $y^\star = (\frac{\bar{y} + \ubar{y}}{2} + l)$kWh  for some $l = 0,\ldots,5$. Considering a larger search space did not lead to significantly better out-of-sample performance. Every tuple~$(p^\star, y^\star)$ encodes a different bidding strategy. Given a training dataset comprising one year of frequency measurements and market prices, we compute the cumulative profit of each strategy via the backtesting procedure outlined above, and we choose the tuple~$(p^\star,y^\star)$ that corresponds to the winning strategy. This strategy is non-anticipative if the year of the training dataset precedes the year of the test dataset. Perhaps surprisingly, we find that from 2017 onward, anticipative parameter tuning has no advantage over non-anticipative tuning. From now on, we thus tune~$p^\star$ and~$y^\star$ non-anticipatively using the year of training data immediately prior to the test dataset. 

\subsection{Experiments: Set-up, Results and Discussion}\label{sec:restuls-discussion}

In the remainder, we distinguish six different simulation scenarios. The nominal scenario uses the parameters of Table~\ref{tab:Params} in Appendix~\ref{Apx:Data} for both training and testing. All other scenarios are based on slightly modified parameters. Specifically, we consider a lossless energy conversion scenario, which trains and tests the baseline strategy under the assumption that $\eta^+=\eta^-=1$. A variant of this scenario assumes lossless energy conversion in training but tests the resulting strategy under the nominal values of $\eta^+$ and $\eta^-$. We also consider two scenarios with weaker robustness guarantees that replace the uncertainty set~$\D$ in the training phase with its subset~$\hat \D$. The resulting bidding strategy can be infeasible because it may fail to  provide the legally required amount of reserve power. The two scenarios do not differ in training but impose different sanctions for infeasibilities in testing. In the first of the two scenarios the vehicle owner is immediately excluded from the reserve market upon the first infeasibility, thus loosing the opportunity to earn money by offering grid services for the rest of the year. In the second scenario, the vehicle owner is penalized by $k_\mathrm{pen} \cdot p^a(t)$ with $k_\mathrm{pen}= 5$ for energy that is missing for frequency regulation (see also Section~\ref{sec:backtest}) and by $ p^y = 0.75$\EUR{}/kWh for energy that is missing for driving. Finally, we consider a scenario in which the vehicle is only equipped with a unidirectional charger, that is, we set $\bar y^-=0$. Thus, the vehicle is unable to feed power back into the grid. Requests for up-regulation ($\delta(t)<0)$ can therefore only be satisfied by consuming less energy, which is possible only if $\delta (t) x^r(t)\leq x^b(t)$. 

\begin{figure}[!t]
	\centering
	\includegraphics[trim=0cm 0cm 0cm 0.0cm, clip, width=\linewidth]{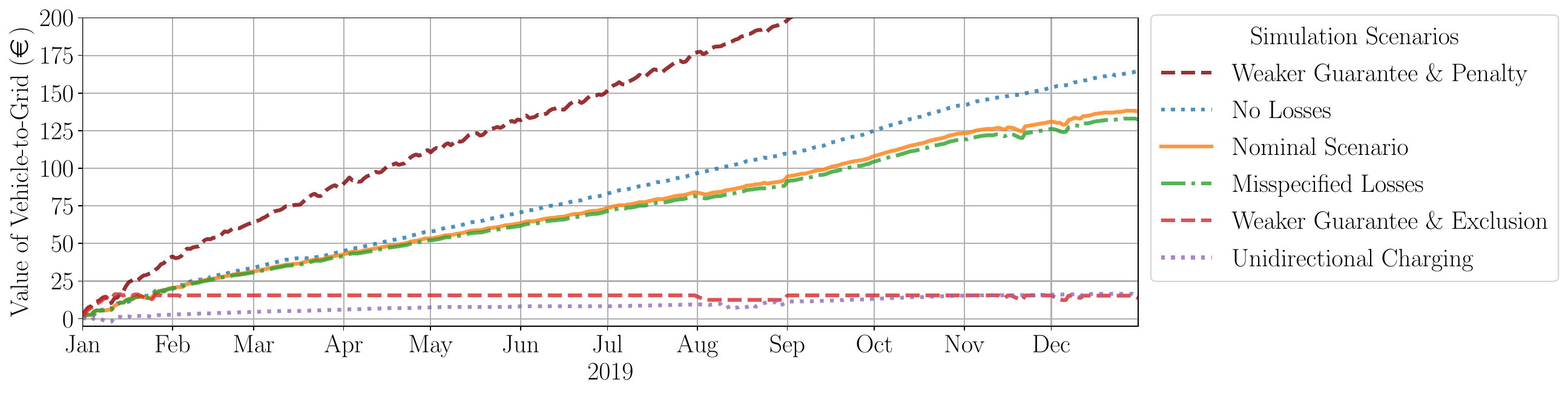}
	\caption{Value of vehicle-to-grid in 2019 under different simulation scenarios.}
	\label{fig:Results}
\end{figure}

Figure~\ref{fig:Results} visualizes the value of vehicle-to-grid in 2019 as a function of time for each of the six simulation scenarios. We first observe that the value of vehicle-to-grid in the lossless energy conversion scenario amounts to 165\EUR{} at the end of the year and thus significantly exceeds the respective value of 138\EUR{} in the nominal scenario. Using a perfectly efficient vehicle charger would thus have boosted the value of vehicle-to-grid by~$20$\% in 2019. This is not surprising because a perfect charger prevents costly energy conversion losses. Note also that the scenario with misspecified efficiency parameters results in almost the same value of vehicle-to-grid as the nominal scenario, which suggests that misrepresenting~$\eta^+$ or~$\eta^-$ in training has a negligible effect on the test performance. However, the underlying bidding strategy is~\emph{not} guaranteed to be feasible because it neglects energy conversion losses. Even though this strategy happens to remain feasible throughout~2019, it bears the risk of financial penalties or market exclusion. The two bidding strategies with weakened robustness guarantees initially reap high profits by aggressively participating in the reserve market, but they already fail in the first half of January to fulfill all market commitments. If infeasibilities are sanctioned by market exclusion, the cumulative excess profit thus remains flat after this incident. If infeasibilities lead to financial penalties, on the other hand, the cumulative excess profit drops sharply below zero near the incident but recovers quickly and then continues to grow steadily. As only a few other mild infeasibilities occur in 2019, the end-of-year excess profit of this aggressive bidding strategy still piles up to 296\EUR, which is more than twice the excess profit in the nominal scenario.
We conclude that the current level of financial penalties is too low to deter vehicle owners from making promises they cannot honor. Finally, with a unidirectional charger, the 2019 value of vehicle-to-grid falls to $15$\EUR{}, which is only~$11$\% of the respective value with a bidirectional charger. This is not too surprising given that the unidirectional vehicle can only offer as much frequency regulation as the power it buys from the utility, which is all the up-regulation it could ever provide.

In terms of battery usage, the nominal vehicle underwent 66.5 charging cycles for driving and an additional 18.9 charging cycles for frequency regulation in 2019, which suggests increased battery degradation. However, when delivering frequency regulation the standard deviation of the state-of-charge from its midpoint of 50\% reduced from 9.5 percentage points to 7.9 percentage points, which is generally thought to benefit battery longevity.

\begin{figure}[!t]
	\centering
	\includegraphics[trim=0cm 0cm 0cm 0.0cm, clip, width=\linewidth]{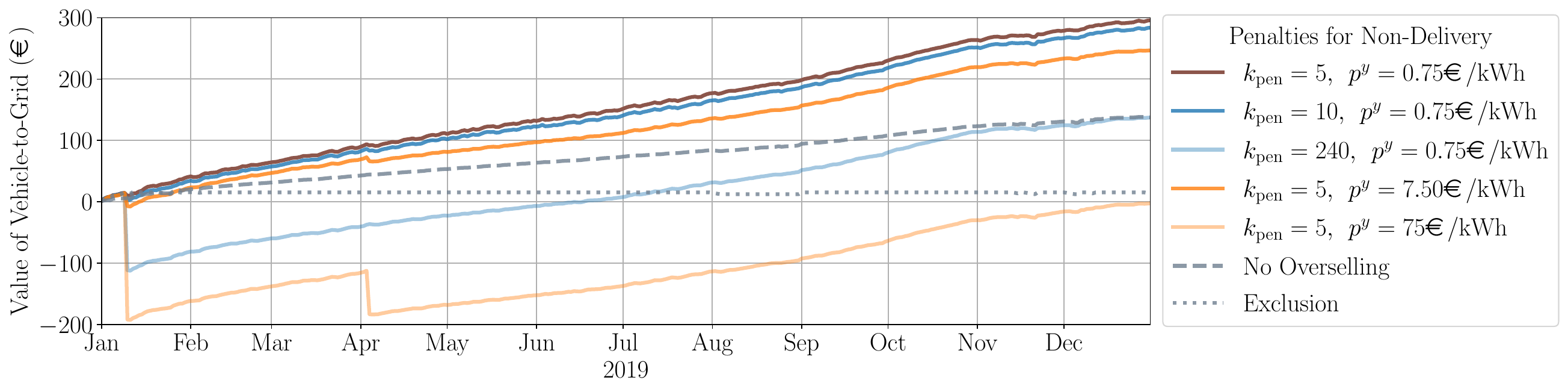}
	\caption{Value of vehicle-to-grid in 2019 versus penalties for non-delivery.}
	\label{fig:Penalties}
\end{figure}

As the bidding strategy with a weakened robustness guarantee can earn high profits when infringements of the EU delivery guarantee incur only financial penalties, we carry out an additional experiment to analyze the impact of the penalty parameters~$k_{\mathrm{pen}}$ and~$p^y$ on the value of vehicle-to-grid; see Figure~\ref{fig:Penalties}. We observe that for $p^y = 0.75$\EUR/kWh, doubling the penalty factor to~$k_{\mathrm{pen}}=10$ decreases the value of vehicle-to-grid by only~$4.1$\% to~$284$\EUR{}. An additional calculation reveals that the TSO would have to set~$k_{\mathrm{pen}}\approx 240$ in order push the value of vehicle-to-grid below that attained in the basline scenario, which fulfills the EU delivery guarantee. On the other hand, for~$k_{\rm pen}=5$, a tenfold increase of~$p^y$ to $7.50$\EUR/kWh decreases the value of vehicle-to-grid by~$16.6$\% to $247$\EUR{}, and an additional tenfold increase of~$p^y$ to~$75$\EUR/kWh pushes the value of vehicle-to-grid below zero. Increasing~$p^y$ from $0.75$\EUR/kWh to $7.50$\EUR/kWh also reduces the number of days on which the vehicle owner does not have enough energy to drive from~7 to~2.

\begin{figure}[!t]
	\centering
	\includegraphics[trim=0cm 0cm 0cm 0.0cm, clip, width=6in]{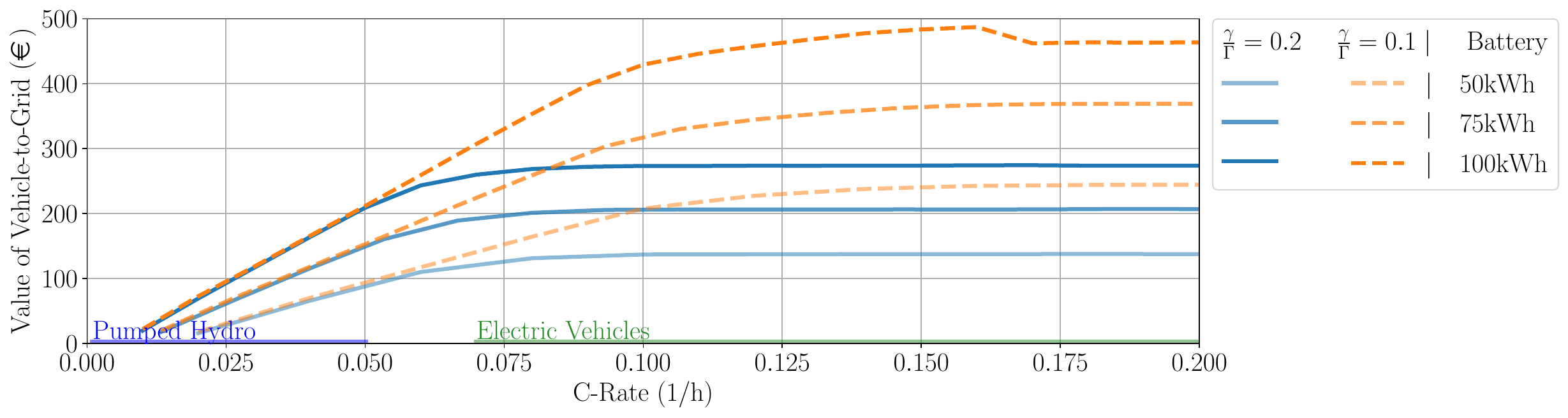}
	\caption{Value of vehicle-to-grid in 2019 as a function of C-rate and activation ratio.}
	\label{fig:c-rate}
\end{figure}

In the next experiment we investigate how the value of vehicle-to-grid depends on the activation ratio~$\gamma/\Gamma$ and the battery's size and charge rate (C-rate). The C-rate is defined as the ratio of the charger power and the battery size, and thus it expresses the percentage of the battery's total capacity that can be charged within one hour. Figure~\ref{fig:c-rate} shows that the value of vehicle-to-grid increases with the C-rate up to a saturation point that is insensitive to the battery size but decreases with the activation ratio. In the saturation regime, the value of vehicle-to-grid increases with the battery size. Typical electric vehicles can be fully charged overnight, within about 8~hours. The corresponding C-rate of~$0.125\mathrm{h}^{-1}$ falls within the saturation regime for both investigated activation ratios of~$0.1$ and~$0.2$. This observation has two implications. First, for typical electric vehicles the value of vehicle-to-grid cannot be increased by increasing the charger power (and thereby increasing the C-rate). This insight contradicts previous studies by \cite{WK05a}, \cite{PC15} and \cite{OB19}, which advocate for electric vehicles with C-rates of 1, 0.45, and 0.37, respectively. The reason for this discrepancy is that none of the previous studies faithfully account for the EU delivery guarantee. In particular, \cite{OB19} allows the vehicle owner to anticipate future frequency deviations, and \cite{PC15} assume that the bidding strategy can be updated on an hourly basis, thereby exploiting information that is not available at the the time when the market bids are collected by the TSO (\ie, one day in advance). By underestimating the amount of energy that the vehicle owner must be able to exchange with the TSO to satisfy all future obligations on the reserve market, these studies overestimate the amount of regulation power that can be sold, which makes potent battery chargers appear more useful than they actually are. The second implication is that the activation ratio has a critical impact on the value of vehicle-to-grid. The incumbent storage providers of the European electricity grid, namely  pumped-hydro storage power plants, have C-rates of about~$0.0125\mathrm{h}^{-1}$ \citep{EC20}, which are significantly smaller than the C-rates of electric vehicles. At such C-rates, the value of providing frequency regulation is the same for activation ratios of~0.1 and~0.2. To minimize competition from vehicle-to-grid, pumped-hydro storage operators may thus have an incentive to lobby for high activation ratios.

From the perspective of a TSO, the higher the activation ratio, the larger the uncertainty set~$\D$ and the lower the probability of blackouts. However, Figure~\ref{fig:Daily_StdDev} suggests that an activation period of $30$~minutes is already conservative. On the other hand, the larger~$\D$, the harder it is for storage operators to provide regulation power, which may lead to less competition, higher market prices, and a higher total cost of frequency regulation. As the system operator is a public entity, this cost is ultimately borne by the public, and the choice of the activation ratio is a political decision.

\begin{figure}[!t]
	\centering
	\includegraphics[trim=0cm 0cm 0cm 0.0cm, clip, width=6in]{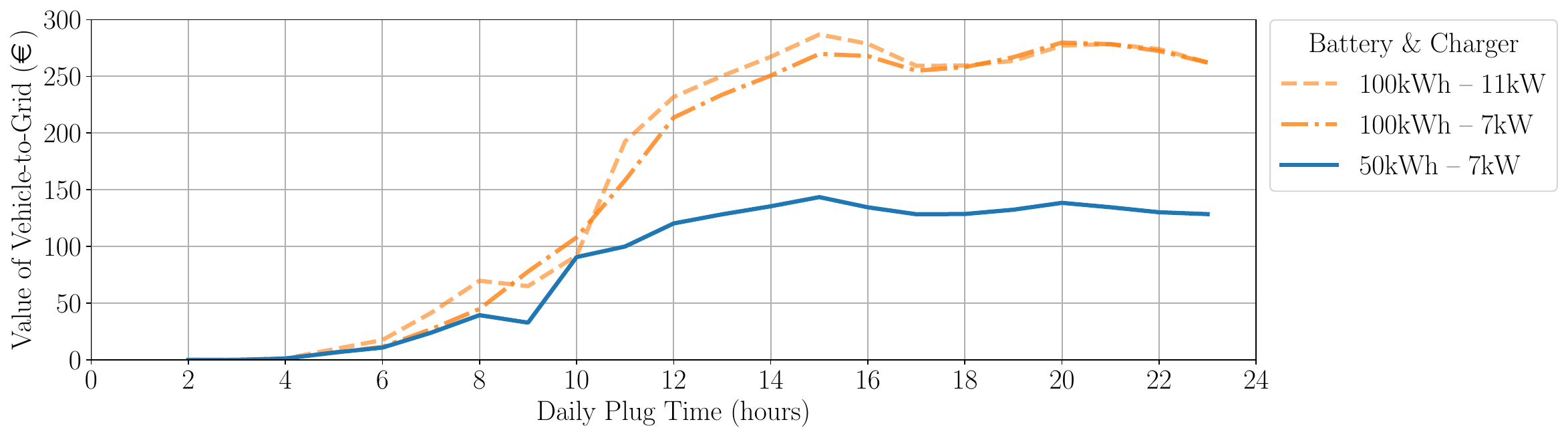}
	\caption{Value of vehicle-to-grid in 2019 versus daily plug time.}
	\label{fig:plug_time}
\end{figure}

Figure~\ref{fig:plug_time} shows the value of vehicle-to-grid in 2019 against the daily plug time, that is, the total amount of time during which the vehicle is connected to a bidirectional charger. In this experiment, we assume that a daily plug time of~$t_p \in [0\text{h}, 24\text{h}]$ means that cars are plugged from midnight to~$\frac{t_p}{2}$~am and from~$12\text{h} - \frac{t_p}{2}$~pm to midnight the next day. Whenever the car is not plugged, it consumes a constant amount of power such that the total consumption over the year corresponds to a mileage of 10,000~km as in the baseline experiment. We observe that the value of vehicle-to-grid increases with the daily plug time and saturates after 15~hours at a level that scales with the battery size. Thus, a daily plug time of 15~hours suffices for offering the maximal possible amount of regulation power. Additional experiments show that for an activation ratio of 0.1 instead of 0.2, the saturation point increases to 20~hours.

\section{Conclusions}
\label{sec:conclusions}

We develop an optimization model for the decision problem of an electricity storage operator offering frequency regulation. To our best knowledge, this is the first model that faithfully accounts for the delivery guarantees required by the European Commission. In contrast, all existing models relax the true delivery guarantee constraints and therefore risk that the electricity storage is empty (full) when a request for up- (down-)regulation arrives. In its original formulation, our model represents a non-convex robust optimization problem with functional uncertain parameters and thus appears to be severely intractable. Indeed, the state-of-the-art methods for solving the deterministic version of this problem for a single frequency deviation scenario in discrete time rely on techniques from mixed-integer linear programming. Maybe surprisingly, however, our robust optimization problem is equivalent to a tractable linear program. This is an exact result and does not rely on any approximations. To our best knowledge, we have thus discovered the first practically interesting class of optimization problems that become significantly easier through robustification. 

As our optimization model faithfully captures effective legislation, it enables us to quantify the true value of vehicle-to-grid. This capability is relevant for understanding the economic incentives of different stakeholders such as vehicle owners, aggregators, equipment manufacturers, and regulators.

As for the vehicle owners, we find that their profits from frequency regulation range from~100\EUR\ to~500\EUR\ per year under typical driving patterns. It seems unlikely that such profits are sufficient for vehicle owners to forego the freedom of using their car whenever they please to. Nevertheless, some vehicle owners may choose to participate in vehicle-to-grid for idealistic reasons such as advancing the energy transition away from fossil fuels. Our numerical results also reveal that the value of vehicle-to-grid saturates at daily plug times above 15~hours. Thus, maximal profits from frequency regulation can be reaped even if the vehicle is disconnected from the grid up to 9~hours per day. This means that vehicle owners participating in vehicle-to-grid still enjoy considerable flexibility as to when to drive, which could help to promote the adoption of vehicle-to-grid.

Our results also have ramifications for aggregators, which pool multiple vehicles to offer regulation power. Indeed, aggregators may allow vehicle owners to reserve their vehicles for up to 9~hours of driving per day without sacrificing profit. This leaves vehicle owners ample freedom and reduces the probability that they exceed their driving slots. Thus, the actual number of vehicles available for frequency regulation at any point in time closely matches its prediction. This allows aggregators to place reserve market bids with small safety margins, which ultimately lowers transaction costs. The online supplement extends our one-vehicle problem formulation to a multi-vehicle formulation.

Equipment manufacturers design and sell bidirectional vehicle chargers. Contrary to previous studies that relax the exact delivery guarantee constraints, we find that the battery size and not the charger power is limiting the profits from frequency regulation. Manufacturers thus have no incentive to produce overly powerful bidirectional chargers.

The electricity system and the society as a whole could benefit significantly from vehicle-to-grid, which harnesses the idle storage capacities of electric vehicles and thereby reduces the need for other sources of flexible electricity supply, such as gas power plants or stationary batteries. This in turn reduces the need for imports of natural gas and critical raw materials, increases the long-term security of electricity supply, and decreases greenhouse gas emissions. Regulators may therefore want to make vehicle-to-grid more attractive by prescribing the availability and delivery prices, defining appropriate delivery guarantee requirements and setting the penalties charged for non-compliance. Our results show that the vehicle owners' profits from frequency regulation decrease with the activation ratio and that current penalties are too low to incentivize vehicle owners to respect the law. Regulators could thus make vehicle-to-grid more attractive by decreasing the activation ratio and thereby relaxing the delivery guarantee requirements. Given that the delivery guarantee in our nominal scenario is very restrictive, this would only slightly decrease the reliability of frequency regulation. If, at the same time, regulators were to increase the penalties for non-compliance, then the reliability of frequency regulation might even increase because more vehicle owners would honor their contractual obligations. Our new model can be used for finding the appropriate level of the penalty. 

\paragraph{Acknowledgments} DL acknowledges fruitful discussions with Nadège Faul, Fran\c{c}ois Colet, Wilco Burghout, Paul Codani, Olivier Borne, Emilia Suomalainen, Ja\^{a}far Berrada, Willett Kempton, Evangelos Vrettos, Sophie Hiriart, and Pierre Guillain, and funding from the Institut VEDECOM. 

\linespread{1}
\small

\bibliographystyle{abbrvnat}
\bibliography{bibliography}

\newpage

\linespread{1.5}
\normalsize

\appendix

\section{Problem Data and Model Parameters}\label{Apx:Data}

\begin{table}[h!]
	\centering
	\caption{Parameters of the Nominal Simulation Scenario.}
	\label{tab:Params}
	\resizebox{0.54\textwidth}{!}{%
	\begin{tabular}{llrc}
		Parameter & Symbol & Value & Unit \\
		\midrule
		\multicolumn{4}{c}{Vehicle Data}\\
		Minimum State-of-Charge & $\ubar{y}$& 10 & kWh \\
		Maximum State-of-Charge & $\bar{y}$ & 40 & kWh \\
		Target State-of-Charge & $y^\star$ & 27 & kWh \\
		Deviation Penalty & $p^\star$ & 0.15 & \EUR/kWh \\
		Charging Efficiency  & $\eta^+$  & 85 & \% \\
		Discharging Efficiency & $\eta^-$ & 85 & \% \\
		Maximum Charing Power & $\bar{y}^+$ & 7 & kW \\
		Maximum Discharging Power & $\bar{y}^-$ & 7 & kW \\
		Yearly Energy for Driving & & 2,000 & kWh \\
		Fraction of Time Driving & & 27 & \% \\
		\midrule
		\multicolumn{4}{c}{Grid Data}\\
		Nominal frequency  & $f_0$ & 50 & Hz \\
		Normalization constant & $\Delta f$ & 200 & mHz\\
		Average Utility Price from 2015 to 2019 & & 14.31 & cts/kWh \\
		Average Availability Price from 2015 to 2019 & & 0.825 & cts/kW/h \\
		\midrule
		\multicolumn{4}{c}{General Parameters}\\
		Trading Interval & $\Delta t$ & 30 & min \\
		Activation Period in $\D$ & $\gamma$ & 30 & min \\
		Regulation Cycle in $\D$ & $\Gamma$ & 2.5 & h \\
		Activation Period in $\hat \D$ & $\hat \gamma$ & 30 & min \\
		Regulation Cycle in $\hat \D$ & $\hat \Gamma$ & 1 & day \\
		Sampling Rate for Frequency Measurements & & 0.1 & Hz \\
		Planning Horizon & $T$ & 1 & day \\
		\bottomrule
	\end{tabular}}
\end{table}

\begin{figure}[!b]
	\centering
	\includegraphics[trim=0cm 0cm 0cm 0.0cm, clip, width=6.5in]{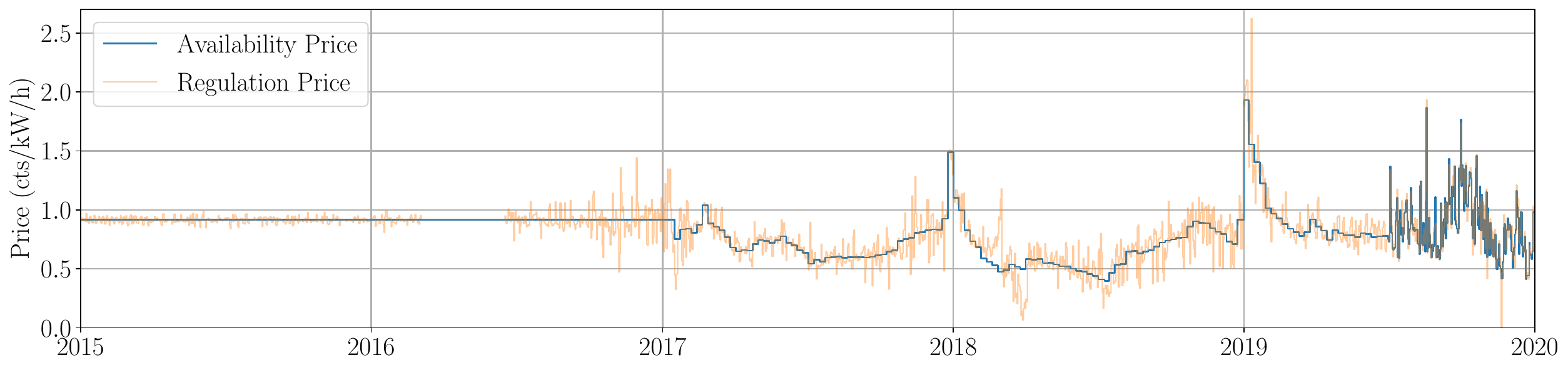}
	\caption{Evolution of the availability and regulation prices from 2015 to 2019. The regulation price changes every~$10$s. For better visibility, we show its daily averages.}
	\label{fig:Reserve price}
\end{figure}

\section{Vehicle Aggregator}\label{sec:aggregator}
  
So far we considered a single vehicle in isolation. In reality, an aggregator would pool hundreds to thousands of electric vehicles to participate in frequency regulation. Aggregators could use our results from the one-vehicle setting as follows. On each day, they ask drivers to schedule their driving needs for the next day, or they infer these needs from past travel patterns. Next, they solve our linear program for each car individually, which can be done efficiently by parallel computing even for large fleets with thousands of vehicles. Last, they add up the regulation power each car can provide and sell it to the grid operator. This approach is easy to implement and requires few computational resources at the expense of sacrificing some optimality. However, it is instructive to formulate a problem for multiple vehicles as it allows us to distinguish between constraints that only apply on the aggregator level and between constraints that apply on the individual vehicle level.

Below, we formulate a planning problem that allows for asymmetric bids of up- and down-regulation on the vehicle level but enforces symmetric bids on an aggregate level. To do so, we replace the decision variable for regulation power~$x^r(t)$ at time $t \in \T$ by a decision variable for up-regulation $x^\uparrow_v(t)$ and down-regulation $x^\downarrow_v(t)$ for each vehicle $v$ in the vehicle portfolio~$\set{V}$. The instantaneous power consumption of a vehicle $v\in\mathcal V$ from the grid is then no longer given by $x_v^b(t) + \delta(t) x_v^r$; instead, it becomes $x_v^b(t) + [\delta(t)]^+ x^\downarrow_v(t) - [\delta(t)]^- x^\uparrow_v(t)$. Consequently, the charging and discharging rates in equations (1a-b) are redefined as
\begin{equation*}
    y^+\left( x^b_v(t), x^{\uparrow}_v(t), x^{\downarrow}_v(t), \delta(t) \right) = \left[ x^b_v(t) + \left[ \delta(t) \right]^+ x^{\downarrow}_v (t) - \left[ \delta(t) \right]^- x^{\uparrow}_v(t) \right]^+
\end{equation*}
and
\begin{equation*}
    y^-\left( x^b_v(t), x^{\uparrow}_v(t), x^{\downarrow}_v(t), \delta(t) \right) = \left[ x^b_v(t) + \left[ \delta(t) \right]^+ x^{\downarrow}_v (t) - \left[ \delta(t) \right]^- x^{\uparrow}_v(t) \right]^-.
\end{equation*}
The state-of-charge function of vehicle~$v$ then becomes
\begin{align*}
    & y_v(x^b_v, x^{\uparrow}_v, x^{\downarrow}_v, \delta, y_{0,v}, t ) \\
    = & y_{0,v} + \int_0^t \eta^+_v y^+\left( x^b_v(t'), x^{\uparrow}_v(t'), x^{\downarrow}_v(t'), \delta(t') \right)- \frac{y^-\left( x^b_v(t'), x^{\uparrow}_v(t'), x^{\downarrow}_v(t'), \delta(t') \right)}{\eta^-_v} - d_v(t') \, \mathrm{d}t',
\end{align*}
where $y_{0,v}$ refers to the initial state-of-charge of vehicle~$v$. The centralized planning problem of an aggregator pooling multiple vehicles with asymmetric bidding behavior can thus be represented as
\begin{equation}
\arraycolsep=2pt
\tag{A}
\label{pb:C}
\begin{array}{>{\displaystyle}c*3{>{\displaystyle}l}}
\min_{x^b, x^{\uparrow}, x^{\downarrow}} & \multicolumn{3}{>{\displaystyle}l}{ c\Big(\sum_{v \in \set{V}} x^b_v, \sum_{v \in \set{V}} x^{\uparrow}_v \Big) + \sum_{v \in \V} \max_{\delta \in \set{\hat{D}}, \, y_{v,0} \in \set{\hat{Y}}_{v,0}} \varphi_v\big( y(x^b_v, x^{\uparrow}_v, x^{\downarrow}_v, \delta, y_{v,0}, T)\big)} \\
\subj & y^+( x^b_v(t), x^{\uparrow}_v(t), x^{\downarrow}_v(t), \delta(t) ) &\leq \bar{y}^+_v(t) & \forall \delta \in \D,~ \forall t \in \T,~ \forall v \in \set{V} \\
& y^-( x^b_v(t), x^{\uparrow}_v(t), x^{\downarrow}_v(t), \delta(t) ) &\leq \bar{y}^-_v(t) & \forall \delta \in \D,~ \forall t \in \T,~ \forall v \in \set{V} \\
& y_v(x^b_v, x^{\uparrow}_v, x^{\downarrow}_v, \delta, y_{v,0}, t) &\leq \bar{y}_v & \forall \delta \in \D,~\forall t \in \T,~ \forall v \in \set{V},~ \forall y_{v,0} \in \set{Y}_{v,0} \\
& y_v(x^b_v, x^{\uparrow}_v, x^{\downarrow}_v, \delta, y_{v,0}, t) &\geq \ubar{y}_v & \forall \delta \in \D,~ \forall t \in \T,~ \forall v \in \set{V},~ \forall y_{v,0} \in \set{Y}_{v,0} \\
& x^b_v,x^\uparrow_v,x^\downarrow_v \in \set{X} && \forall v \in \set{V} \\
& \sum_{v \in \set{V} x^{\uparrow}_v(t) - x^{\downarrow}_v(t)} & = 0 &\forall t \in \set{T},
\end{array}
\end{equation}
where the equality constraint ensures that the aggregator bids the same amount of up- and down-regulation to the grid operator. Much like the original problem~\eqref{pb:Rc} for a single vehicle, the aggregator's problem~\eqref{pb:C} also admits a tractable linear programming reformulation.

\begin{Th}[Aggregator]\label{th:c}
    Problem~\eqref{pb:C} admits an exact linear programming reformulation.
\end{Th}

The linear program equivalent to~\eqref{pb:C} scales linearly with the number of vehicles. The only constraint coupling the individual vehicles is the equality constraint in problem~\eqref{pb:C}. For a limited number of vehicles, the aggregator will be able to solve the reformulation directly. However, the linear program may become too large to fit in memory or be solved in a reasonable amount of time for large vehicle fleets. In this case, the optimization effort can be distributed across vehicles via decomposition methods such as the alternating direction method of multipliers~\citep{SB11}. Since the coupling constraint only involves the bids for up- and down-regulation, each vehicle would iteratively solve a version of problem~\eqref{pb:C} \emph{without} the coupling constraint and communicate the resulting bids for up- and down-regulation. The aggregator would compute the mismatch between the total bids for up- and down-regulation from all vehicles and broadcast a mismatch signal to all vehicles, which then recompute their market bids. This distributed approach has the added advantage that it does not communicate potentially sensitive private information such as estimated future driving patterns to the aggregator.

Numerically, when pairing two nominal vehicles from Section~\ref{sec:restuls-discussion}, one with a bidirectional charger and one with a unidirectional charger, in problem~\eqref{pb:C}, the value of vehicle-to-grid raises by~38\% from 152\EUR to 210\EUR for both vehicles together, which shows the benefit of imposing the balance between up- and down-regulation only at the aggregate level. To demonstrate the scalability of our formulation, we now solve problem~\eqref{pb:C} for a fleet of 9 different vehicle types with 100 vehicles of each type. Two vehicle types correspond the nominal vehicle with and without a bidirectional charger. The other 7 vehicle types are all variants of the nominal vehicle with the same charger power, a battery size of either $50$kWh or $75$kWh, and vehicle usage during either workdays, weekends, or exclusively from June through August for seasonal vehicles, such as summer rental cars. In total, $4$ out of the $9$ vehicles types have unidirectional chargers. We find that the fleet of $900$~vehicles offers $1.40$MW of regulation power on average throughout the year $2019$, with a resulting value of vehicle-to-grid of $98,473$\EUR. The profit per MW of frequency regulation is thus about $70,337$\EUR.

Beyond asymmetric regulation bids, aggregators can decide which cars (or other flexible electric devices) to use to provide a given amount of regulation power. This has two advantages. First, battery degradation can be reduced by managing the state-of-charge of vehicles more precisely. Second, charging and discharging losses can be reduced by running vehicle chargers at their nominal operating points. In addition, aggregators may be able to trade on intraday markets, which would allow them to offer more regulation power for a given aggregate battery size. For example, if they encounter an extreme deviation trajectory, they can buy or sell electricity on the intraday market to maintain the state-of-charge within its bounds. This is risky, however, because intraday markets may not be liquid enough for aggregators to find trading partners. Especially not when the grid is already in distress, which is the case when extreme frequency deviations occur.

\section{Proofs}\label{Apx:Proofs}


This online supplement contains the proofs of all theorems in the main text as well as some auxiliary propositions that are needed for these proofs.

\begin{Prop}\label{Prop:y}
	Holding all other factors fixed, the battery state-of-charge $y(x^b,x^r,\delta,y_0,t)$ is concave nondecreasing in $x^b$, concave in $x^r$, concave nondecreasing in $\delta$, and affine nondecreasing in $y_0$.
\end{Prop}

\begin{proof}
	By definition we have
	\begin{align*}
		y\left(x^b,x^r,\delta,y_0,t\right) = & y_0 + \int_{0}^{t} \eta^+ \left[x^b(t') + \delta(t')x^r(t')\right]^+ - \frac{\left[x^b(t') + \delta(t')x^r(t')\right]^-}{\eta^-} -d(t') \, \mathrm{d} t' \\
		= & y_0 + \int_0^t \min\left\{ \eta^+\left( x^b(t') + \delta(t') x^r(t') \right), \frac{x^b(t') + \delta(t')x^r(t')}{\eta^-} \right\} - d(t') \, \mathrm{d}t',
	\end{align*}
	where the second equality holds because $\eta^+ < \frac{1}{\eta^-}$. The postulated properties of $y(x^b,x^r,\delta,y_0,t)$ follow from the observation that the minimum of two (nondecreasing) affine functions is a concave (nondecreasing) function~\cite[p.~73]{SB04}.
\end{proof}

\begin{Prop} If $\delta \in \D$, then $\mathrm{Var}(\delta) \leq\ceil{T/\Gamma} \gamma/T$.
\end{Prop}\label{Prop:Var}

\begin{proof}
	We will show that
	$ \max_{\delta \in \D} \mathrm{Var}(\delta) \leq \beta/T$, where $\beta = \ceil{T/\Gamma} \gamma$.
	To this end, we note that
	\begin{equation}
		\label{eq:Var}
		\max_{\delta \in \D} \mathrm{Var} (\delta) = \max_{\delta \in \D^+} \mathrm{Var} (\delta) \leq \max_{\delta \in \D_\beta} \mathrm{Var}(\delta),
	\end{equation}
	where $\D_\beta = \{ \delta \in \mathcal{L}(\mathcal{T}, [0,1]) : \int_{0}^{T} \delta(t) \, \mathrm{d}t \leq \beta \}$. The equality in~\eqref{eq:Var} holds because $\text{Var}(\delta)$ remains unchanged when~$\delta(t)$ is replaced with~$\vert \delta(t) \vert$ for every~$t \in \T$. Moreover, the inequality holds because $\D^+ \subseteq \D_\beta$. To see this, observe that for any $\delta \in \D^+$ we have
	\begin{equation*}
		\int_0^T \delta(t) \, \mathrm{d}t = 
		\sum_{n=1}^{\lceil T/\Gamma \rceil } \int_{(n-1)\Gamma}^{\min\{n\Gamma, T\}} \delta(t) \, \mathrm{d} t
		\leq \ceil*{\frac{T}{\Gamma}} \gamma = \beta.
	\end{equation*}
	Note that $\D_\beta$ constitutes a budget uncertainty set of the type introduced by~\cite{DB04}, where the uncertainty budget $\beta$ corresponds to the maximum number of regulation cycles within the planning horizon multiplied by the duration of an activation period. Thus,~$\beta$ can be viewed as the maximum amount of time within the planning horizon during which all reserve commitments must be honored. By weak duality, the highest variance of any scenario in $\D_\beta$ satisfies
	\begingroup
	\allowdisplaybreaks
	\begin{align*}
		\max_{\delta \in \D_\beta} \mathrm{Var}(\delta)
		\leq & \min_{\lambda \geq 0} \max_{\delta \in \set{L}(\T,[0,1])} \frac{1}{T} \int_{0}^{T}\delta^2(t) \, \mathrm{d}t + \lambda \left( \beta - \int_{0}^{T} \delta(t) \, \mathrm{d}t\right) \\
		= & \min_{\lambda \geq 0} \lambda \beta + \max_{\delta \in \set{L}(\T,[0,1])} \int_{0}^T \delta(t) \left( \delta(t)/T - \lambda \right) \mathrm{d}t \\
		= & \min_{\lambda \geq 0} \lambda \beta + T \max_{\delta \in [0,1]} \delta \left( \delta/T - \lambda \right) \\
		= & \min_{\lambda \geq 0} \lambda \beta + \max\left\{ 0,1-\lambda T \right\} \\
		= & \min_{\lambda \geq 0} \max\left\{ \lambda \beta, 1 - \lambda \left( T - \beta \right) \right\} = \frac{\beta}{T}.
	\end{align*}
	\endgroup
	The claim now follows by substituting the above result into~\eqref{eq:Var} and recalling the definition of~$\beta$.
\end{proof}


\begin{Prop}\label{Prop:D}
	The following statements hold.
	\begin{enumerate}
		\item[(i)] $L\D_{\K} \subseteq \D$ and $L^\dagger \D = \D_{\K}$.
		\item[(ii)] $L\D_{\K}^+ \subseteq \D^+$ and $L^\dagger \D^+ = \D^+_\K$.
	\end{enumerate}
\end{Prop}

\begin{proof}
	As for assertion~\textit{(i)} we first prove that $L \D_\K \subseteq \D$. To this end, select any $\bm{\delta} \in \D_\K$, and define $\delta = L \bm{\delta}$. It is easy to see that $\delta \in \set{L}(\T, [-1,1])$. Next, select any $t \in \T$. If $t \leq \Gamma$, note that
	\begin{equation*}
		\int_{\left[t-\Gamma\right]^+}^t \vert \delta(t') \vert \, \mathrm{d}t'
		= \int_0^t \vert \delta(t') \vert \, \mathrm{d}t'
		\leq \int_0^{\Gamma} \vert \delta(t') \vert \, \mathrm{d}t' = \Delta t \sum_{l=1}^{\Delta k} \vert \delta_l \vert \leq \gamma,
	\end{equation*}
	where the auxiliary parameter $\Delta k = \Gamma/\Delta t$ is integral thanks to Assumption~\ref{Ass:div}. The second inequality in the above expression holds because $\bm{\delta} \in \D_\K$. If $t \geq \Gamma$, on the other hand, we define $k = \lceil \frac{t}{\Delta t} \rceil$ and $\alpha = \lceil \frac{t}{\Delta t} \rceil - \frac{t}{\Delta t} \in [0,1)$. Then, we find
	\begin{align*}
		\int_{[t-\Gamma]^+}^t \vert \delta(t') \vert \, \mathrm{d}t'
		=& \int_{t-\Gamma}^{ (k-\Delta k) \Delta t} \vert \delta(t') \vert \, \mathrm{d}t' + \int_{(k-\Delta k) \Delta t}^{ (k-1) \Delta t} \vert \delta(t') \vert \, \mathrm{d}t' + \int_{(k-1) \Delta t}^{ t} \vert \delta(t') \vert \, \mathrm{d}t' \\
		= & \left(k \Delta t - t \right) \vert \delta_{k-\Delta k} \vert  
		+ \Delta t \sum_{l=k-\Delta k +1}^{k-1} \vert \delta_l \vert
		+ \left( t - \left(k-1\right)\Delta t \right) \vert \delta_{k} \vert \\
		= & \Delta t \bigg( \alpha \sum_{l=k-\Delta k}^{k-1} \vert \delta_l \vert 
		+ \left(1 - \alpha\right) \sum_{l=k- \Delta k + 1}^{k} \vert \delta_l \vert \bigg) \leq \gamma,
	\end{align*}
	where the inequality holds because $\bm{\delta} \in \D_\K$, which ensures that both $\sum_{l=k-\Delta k}^{k-1} \vert \delta_l \vert$ and $\sum_{l=k-\Delta k + 1}^{k} \vert \delta_l \vert$ are smaller or equal to~$\gamma/\Delta t$. As $t \in \T$ was chosen arbitrarily, this implies that $\delta \in L\D_\K$. In summary, we have shown that $L \D_\K \subseteq \D$. 
	
	Next, we show that $L^\dagger \D \subseteq \D_\K$. To this end, select any $\delta \in \D$ and define $\bm{\delta} = L^\dagger \delta$. It is easy to see that $\bm{\delta} \in [-1,1]^K$. Moreover, for any $k \in \K$ we have 
	\begin{equation*}
		\sum_{l = 1 + [k - \Gamma/\Delta t]^+}^k \delta_l = \sum_{l = 1 + [k - \Gamma/\Delta t]^+}^k \frac{1}{\Delta t}\int_{\T_l} \delta(t') \, \mathrm{d}t' = \frac{1}{\Delta t} \int_{[k \Delta t - \Gamma]^+}^{k \Delta t}  \delta(t') \, \mathrm{d}t' \leq \frac{\gamma}{\Delta t},
	\end{equation*}
	where the inequality holds because $\delta \in \D$. As $k \in \K$ was chosen arbitrarily, this implies that $\bm{\delta} \in L^\dagger \D$. In summary, we have shown that $L^\dagger \D \subseteq \D_\K$.
	
	Finally, we prove that $\D_\K \subseteq L^\dagger \D$. To this end, we observe that $L^\dagger L$ coincides with the identity mapping on $\mathbb{R}^K$. As $L \D_\K \subseteq \D$, this implies that 
	\begin{equation*}
		\D_\K = L^\dagger L \D_\K \subseteq L^\dagger \D.
	\end{equation*}
	Since both $L^\dagger \D \subseteq \D_\K$ and $\D_\K \subseteq L^\dagger\D$, we have in fact shown that $L^\dagger \D = \D_\K$. Thus assertion~\textit{(i)} follows. Assertion~\textit{(ii)} can be proved in a similar manner. Details are omitted for brevity.
\end{proof}

\begin{Prop}\label{Prop:yk}
	Holding all other factors fixed, $y_k(\bm{x^b},\bm{x^r},\bm{\delta},y_0)$ is concave nondecreasing in $\bm{x^b}$, concave in $\bm{x^r}$, concave nondecreasing in $\bm{\delta}$, and linear nondecreasing in $y_0$ for any $k \in \K$.
\end{Prop}

\begin{proof}
	The proof widely parallels that of Proposition~\ref{Prop:y} and is therefore omitted.
\end{proof}

The proof of Theorem~\ref{th:time} relies on Propositions~\ref{Prop:yc}--\ref{Prop:H} below.

\begin{Prop}\label{Prop:yc} The following equivalences hold.
\begin{subequations}
\begin{align*}
	  (i) \quad & y^+\left(x^b(t),x^r(t),\delta(t)\right) \leq \bar{y}^+(t) ~ \forall \delta \in \D, \forall t \in \T
	  \iff y^+_k\left(x^b_k,x^r_k,\delta_k\right) \leq \bar{y}^+_k ~ \forall \bm{\delta} \in \D_\K, \forall k \in \K \\
	  (ii) \quad & y^-\left(x^b(t),x^r(t),\delta(t)\right) \leq \bar{y}^-(t) ~ \forall \delta \in \D, \forall t \in \T
	  \iff y^-_k\left(x^b_k,x^r_k,\delta_k\right) \leq \bar{y}^-_k ~ \forall \bm{\delta} \in \D_\K, \forall k \in \K
\end{align*} 
\end{subequations}
\end{Prop}

\begin{proof}
	Assertion~$(i)$ can be reexpressed as 
	\begin{equation*}
		\max_{t \in \T} \max_{\delta \in \D} y^+\left(x^b(t),x^r(t),\delta(t)\right) - \bar{y}^+(t) \leq 0 \iff \max_{k \in \K} \max_{\bm{\delta} \in \D_\K} y^+_k\left(x^b_k,x^r_k,\delta_k\right) - \bar{y}^+_k \leq 0.
	\end{equation*}
	We will prove this equivalence by showing that the left hand sides of the two inequalities are equal. Indeed, a direct calculation reveals that
	\begin{equation}
		\label{eq:yc}
		\begin{aligned}
			\max_{t \in \T} \max_{\delta \in \D} y^+\left(x^b(t),x^r(t),\delta(t)\right) - \bar{y}^+(t) = &
			\max_{t \in \T} \max_{-1 \leq \delta(t) \leq 1} \left[ x^b(t) + \delta(t) x^r(t) \right]^+ - \bar{y}^+(t) \\
			= & \max_{t \in \T} x^b(t) + x^r(t) - \bar{y}^+(t) \\
			= & \max_{k \in \K} x^b_k + x^r_k - \bar{y}^+_k \\
			= & \max_{k \in \K} \max_{-1 \leq \delta_k \leq 1} \left[x^b_k + \delta_k x^r_k\right]^+ - \bar{y}^+_k \\
			= & \max_{k \in \K} \max_{\bm{\delta} \in \D_\K} y^+\left(x^b_k,x^r_k,\delta_k\right) - \bar{y}^+_k,
		\end{aligned}
	\end{equation}
	where the first equality follows from the definition of $y^+$ in~\eqref{eq:y+} and the observation that $\{ \delta(t) : \delta \in \D \} = [-1,1]$, while the second equality holds because~$x^b(t) \geq 0$ and $x^r(t) \geq 0$ which implies that $\delta(t) = 1$ maximizes the instantaneous charging rate. The third equality exploits our assumption that $x^b, x^r$, and~$\bar{y}^+$ are piecewise constant functions. The fourth equality holds because~$x^b_k \geq 0$ and~$x^r_k \geq 0$, which implies that $\delta_k = 1$ maximizes the per-period charging rate. The fifth equality follows again from the definition of $y^+$ in~\eqref{eq:y+} and the observation that $\{\delta_k : \bm{\delta} \in \D_\K\} = [-1,1]$.
	
	The proof of assertion~$(ii)$ is similar and therefore omitted.
\end{proof}

\begin{Prop}\label{Prop:ybar}
	The following equivalence holds.
	\begin{equation*}
	 y\left(x^b,x^r,\delta,y_0,t\right) \leq \bar{y}~\forall \delta \in \D, \forall t \in \T
	 \iff y_{k}\left(\bm{x^b},\bm{x^r},\bm{\delta},y_0\right) \leq \bar{y}~\forall \bm{\delta} \in \D_\K,~\forall k \in \K\cup\{0\}
	\end{equation*}
\end{Prop}

\begin{proof}
	The claim follows if we can show that 
	\begin{equation}\label{eq:ybar}
		\max_{t \in \T} \max_{\delta \in \D} y\left(x^b,x^r,\delta,y_0,t\right) = \max_{k\in\K \cup \{0\}} \max_{\bm{\delta} \in \D_\K} y_{k}\left(\bm{x^b},\bm{x^r},\bm{\delta},y_0\right).
	\end{equation}
	To this end, assume first that $t=k\Delta t$ for some $k \in \K \cup \{0\}$. In this case, we have
	\begin{equation}
		\label{eq:ybar_full}
		\begin{aligned}
			\max_{\delta \in \D} y\left(x^b,x^r,\delta,y_0,t\right)
			= & \max_{\delta \in \D^+} y\left(x^b,x^r,\delta,y_0,t\right) \\
			= & y_0 + \max_{\delta \in \D^+} \int_0^{t} \eta^+ \left( x^b(t') + \delta(t') x^r(t') \right) - d(t') \, \mathrm{d}t' \\
			=& y_0 + \max_{\delta \in \D^+} \sum_{l=1}^{k} \int_{\T_l} \eta^+ \left( x^b_l + \delta(t') x^r_l \right) - d_l \, \mathrm{d}t' \\
			=& y_0 + \max_{\delta \in \D^+} \Delta t \sum_{l=1}^{k} \eta^+ \left( x^b_l + (L^\dagger\delta)_l x^r_l \right) - d_l\\
			=& y_0 + \max_{\bm{\delta} \in \D^+_\K} \Delta t \sum_{l=1}^{k} \eta^+ \left( x^b_l + \delta_l x^r_l \right) - d_l\\
			= & \max_{\bm{\delta} \in \D^+_\K} y_{k}\left(\bm{x^b},\bm{x^r},\bm{\delta},y_0\right)
			= \max_{\bm{\delta} \in \D_\K} y_{k}\left(\bm{x^b},\bm{x^r},\bm{\delta},y_0\right),
		\end{aligned}
	\end{equation}
	where the first equality holds because $\delta \in \D$ if and only if $\vert \delta \vert \in \D^+$ and because $y$ is nondecreasing in~$\delta$ thanks to Proposition~\ref{Prop:y}. The second equality follows from the definitions of $y$, $y^+$, and $y^-$ and from the non-negativity of $x^b$, $x^r$ and $\delta$. The third equality exploits our assumption that $d$, $x^b$ and $x^r$ are piecewise constant. As $\delta$ is integrated against a piecewise constant function, it may be averaged over the trading intervals without changing its objective function value. The fifth equality then follows from Proposition~\ref{Prop:D}, while the sixth equality follows from the definitions of $y_k$, $y^+$ and $y^-$ and from the non-negativity of $\bm{x^b}$, $\bm{x^r}$ and $\bm{\delta}$. The seventh equality, finally, holds because $\bm{\delta} \in \D_\K$ if and only if $\vert \bm{\delta} \vert \in \D^+_\K$ and because $y_k$ is nondecreasing in~$\bm{\delta}$ thanks to Proposition~\ref{Prop:yc}. 
	
	Assume now more generally that $t \in \T_k$ for some $k \in \K$. If the vehicle is driving in trading interval~$\T_k$, then $\bar{y}^+(t) = \bar{y}^-(t) = 0$ for all $t\in \T_k$. Thus, we have
	\begin{align*}
		\max_{\delta \in \D} y\left(x^b,x^r,\delta,y_0,t\right) = & \max_{\delta \in \D} y\left(x^b,x^r,\delta,y_0,(k-1)\Delta t\right) - \int_{(k-1) \Delta t}^{t} d(t') \, \mathrm{d}t' \\
		\leq &  \max_{\delta \in \D} y\left(x^b,x^r,\delta,y_0,(k-1)\Delta t\right) = \max_{\bm{\delta} \in \D_\K} y_{k-1}\left(\bm{x^b},\bm{x^r},\bm{\delta},y_0\right) ~\forall t \in \T_k,
	\end{align*}
	where the inequality holds because $d(t) \geq 0$ for all $t\in \T_k$, and the second equality follows from the first part of the proof. Alternatively, if the vehicle is parked in trading interval~$\T_k$, then $d(t) = 0$ for all $t\in \T_k$. Thus, we have 
	\begingroup
	\allowdisplaybreaks
	\begin{align*}
		\max_{\delta \in \D} y\big(x^b,x^r,\delta,y_0,t\big) &
		= \max_{\delta \in \D^+} y\big(x^b,x^r,\delta,y_0,t\big) \\
		&= \max_{\delta \in \D^+} y\big(x^b,x^r,\delta,y_0,(k-1)\Delta t\big) + \int_{(k-1) \Delta t}^{t} \eta^+ y^+\big(x^b(t'),x^r(t'),\delta(t')\big) \, \mathrm{d}t' \\
		& \leq \max_{\delta \in \D^+} y\big(x^b,x^r,\delta,y_0,k\Delta t\big) = \max_{\bm{\delta} \in \D_\K^+} y_{k}\big(\bm{x^b},\bm{x^r},\bm{\delta},y_0\big) = \max_{\bm{\delta} \in \D_\K} y_{k}\big(\bm{x^b},\bm{x^r},\bm{\delta},y_0\big)
	\end{align*}
	\endgroup
	for all $t \in \T_k$, where the inequality holds because the integral is nondecreasing in~$t$, and the equalities follow from the first part of the proof. In summary, we have shown that
	\begin{equation*}
		\max_{\delta \in \D} y\left(x^b,x^r,\delta,y_0,t\right) \leq \max\left\{ \max_{\bm{\delta} \in \D_\K} y_{k-1}\left(\bm{x^b},\bm{x^r},\bm{\delta},y_0\right), \max_{\bm{\delta} \in \D_\K} y_{k}\left(\bm{x^b},\bm{x^r},\bm{\delta},y_0\right) \right\}
	\end{equation*}
	for all $t \in \T_k$ and $k \in \K$. This implies that 
	\begin{equation*}
		\max_{t \in \T} \max_{\delta \in \D} y\left(x^b,x^r,\delta,y_0,t\right) \leq \max_{k \in \K \cup \{0\}} \max_{\bm{\delta} \in \D_\K} y_{k}\left(\bm{x^b},\bm{x^r},\bm{\delta},y_0\right).
	\end{equation*}
	On the other hand, we have
	\begin{equation*}
		\max_{t \in \T} \max_{\delta \in \D} y\left(x^b,x^r,\delta,y_0,t\right) \geq 
		\max_{k \in \K \cup \{0\}} \max_{\delta \in \D} y\left(x^b,x^r,\delta,y_0,k\Delta t\right) =
		\max_{k \in \K \cup \{0\}} \max_{\bm{\delta} \in \D_\K} y_{k}\left(\bm{x^b},\bm{x^r},\bm{\delta},y_0\right),
	\end{equation*}
	where the equality follows from the first part of the proof. Combining the above inequalities implies~\eqref{eq:ybar}, and thus the claim follows.
\end{proof}

\begin{Prop}\label{Prop:yubar}
	The following equivalence holds.
	\begin{equation*}
	y(x^b,x^r,\delta,y_0,t) \geq \ubar{y}~\forall \delta \in \D, \forall t \in \T
	\iff y_{k}(\bm{x^b},\bm{x^r},\bm{\delta},y_0) \geq \ubar{y}~\forall \bm{\delta} \in \D_\K, \forall k \in \K \cup \{0\}
	\end{equation*}
\end{Prop}

The proof of Proposition~\ref{Prop:yubar} is significantly more challenging than that of Proposition~\ref{Prop:ybar} because $y(x^b,x^r,\delta,y_0,t)$ is concave in~$\delta$. We make it more digestible by first proving two lemmas.

\begin{lem}\label{lem:LDR_Opt}
	If $f:\mathbb{R} \times \T \to \mathbb{R}$ is concave, continuous and nonincreasing in its first argument and piecewise constant on the trading intervals $\T_k$, $k \in \K$, in its second argument, then
	\begin{equation*}
		\min_{\delta \in \D^+} \int_{0}^{t} f\left(\delta(t'),t'\right) \, \mathrm{d}t' = \min_{\delta \in \D^+ \cap \set{L}(\T, \{0,1\})} \int_{0}^t f\left(\delta(t'),t'\right) \, \mathrm{d}t'\quad \forall t \in \T.
	\end{equation*}
\end{lem}

\begin{proof}
	For ease of exposition, assume first that~$t = T$ and define $t_k = \Delta t (k - 1)$ for every $k \in \K$. For every approximation parameter $N \in \mathbb{N}$ we define $\set{N} = \{1,\ldots,N\}$ and set
	\begin{equation*}
		\T_{k,n}^N = \left[\Delta t \left(k - 1 + \frac{n-1}{N}\right), \Delta t \left(k - 1 + \frac{n}{N}\right) \right) \quad \forall k \in \K, ~ \forall n \in \set{N}.
	\end{equation*}
	Note that the $\T^N_{k,n}$, $n \in \set{N}$, are mutually disjoint and that their union coincides with the $k$-th trading interval $\T_k$. Next, introduce a lifting operator $L_N: \mathbb{R}^{K \times N} \to \set{L}(\T, \mathbb{R})$ defined through $(L_N \bm{\delta})(t) = \delta_{k,n}$ if $t \in \T^N_{k,n}$ for $k \in \K$ and $n \in \set{N}$. In addition, let $L^\dagger_N : \set{L}(\T,\mathbb{R}) \to \mathbb{R}^{K \times N}$ be the corresponding adjoint operator defined through $(L^\dagger_N \delta)_{k,n} = \frac{N}{\Delta t} \int_{\T^N_{k,n}} \delta(t) \, \mathrm{d}t$ for $k \in \K$ and $n \in \set{N}$. Using this notation, we first prove that
	\begin{equation}\label{eq:Rie_Int}
		\lim_{N \to \infty} \sum_{k \in \K} \sum_{n \in \set{N}} f\left( (L^\dagger_N \delta)_{k,n}, t_k \right) \frac{\Delta t}{N} = \int_0^T f\left(\delta(t),t\right) \, \mathrm{d}t
	\end{equation}
	for any fixed $\delta \in \D^+$. As $f$ is continuous and nonincreasing in its first argument and piecewise constant in its second argument, we have
	\begin{align*}
		\inf_{t \in \T^N_{k,n}} f\left( \delta(t), t \right) = f\bigg( \sup_{t \in \T^N_{k,n}} \delta(t), t_k \bigg) \leq f\left( \left(L^\dagger_N \delta\right)_{k,n}, t_k \right) \leq f\bigg( \inf_{t \in \T^N_{k,n}} \delta(t), t_k \bigg) = \sup_{t \in \T^N_{k,n}} f\left( \delta(t), t \right)
	\end{align*} 
	for every $k \in \K$, $n \in \set{N}$ and $N \in \mathbb{N}$. Summing over $k$ and $n$ thus yields 
	\begin{equation*}
		\sum_{k \in \K} \sum_{n \in \set{N}} \inf_{t \in \T^N_{k,n}} f\left( \delta(t), t \right) \frac{\Delta t}{N} \leq \sum_{k \in \K} \sum_{n \in \set{N}} f\left( \left(L^\dagger_N \delta \right)_{k,n}, t_k \right) \frac{\Delta t}{N} \leq  \sum_{k \in \K} \sum_{n \in \set{N}} \sup_{t \in \T^N_{k,n}} f\left( \delta(t), t \right) \frac{\Delta t}{N}
	\end{equation*}
	for every $N \in \mathbb{N}$. As $f(\delta(t),t)$ constitutes a composition of a continuous function with a Riemann integrable function, it is also Riemann integrable. Thus, the lower and upper Riemann sums in the above inequality both converge to $\int_0^T f(\delta(t),t) \, \mathrm{d}t$ as $N$ tends to infinity. This observation establishes~\eqref{eq:Rie_Int}. As $\delta \in \D^+$ was chosen arbitrarily, we may thus conclude that
	\begin{equation*}
		\inf_{\delta \in \D^+} \int_{0}^{T} f\left( \delta(t),t\right) \, \mathrm{d}t = \inf_{\delta \in \D^+} \lim_{N \to \infty} \sum_{k \in \K} \sum_{n \in \set{N}} f\left( \left(L^\dagger_N\delta\right)_{k,n}, t_k \right) \frac{\Delta t}{N}.
	\end{equation*}
	For the following derivations we introduce the auxiliary uncertainty set 
	\begin{equation*}
		\D^+_{KN} = \left\{ \bm{\delta} \in \left[-1,1\right]^{KN} : \sum_{l = 1 + [m - N \Gamma/\Delta t]^+}^{m}  \delta_{l} \leq N \frac{\gamma}{\Delta t} \quad \forall m = 1,\ldots,KN \right\}
	\end{equation*}
	for $N \in \mathbb{N}$. By slight abuse of notation, we henceforth naturally identify any matrix $\bm \delta \in \R^{K \times N}$ with the vector obtained by concatenating the rows of $\bm \delta$. This convention allows us, for example, to write $\bm \delta \in \D^+_{KN}$ even if $\bm \delta$ was initially defined as a~$K \times N$-matrix. By repeating the arguments of Proposition~\ref{Prop:D}, it is easy to show that $L_N \D^+_{KN} \subseteq \D^+$ and $L^\dagger_N\D^+ = \D^+_{KN}$ for all $N \in \mathbb{N}$. Using these relations, we will now prove that
	\begin{equation}\label{eq:Lim_Inf}
		\inf_{\delta \in \D^+} \int_{0}^{T} f\left( \delta(t),t \right) \, \mathrm{d}t =  \lim_{N \to \infty} \inf_{\bm{\delta} \in \D^+_{KN}} \sum_{k \in \K} \sum_{n \in \set{N}} f\left( \delta_{k,n}, t_k \right) \frac{\Delta t}{N}.
	\end{equation}
	To this end, select any $\epsilon > 0$ and $\delta^\star \in \D^+ $ with $\int_0^T f(\delta^\star(t),t) \, \mathrm{d}t \leq \inf_{\delta \in \D^+} \int_0^{T} f(\delta(t),t) \, \mathrm{d}t + \epsilon$, and choose $N_\epsilon$ large enough such that 
	\begin{equation*}
		\left \vert \sum_{k \in \K} \sum_{n \in \set{N}} f\left( \left(L^\dagger_N \delta^\star\right)_{k,n}, t_k \right) \frac{\Delta t}{N} - \int_0^T f\left( \delta^\star(t), t \right) \, \mathrm{d}t \right \vert \leq \epsilon \quad \forall N \geq N_\epsilon.
	\end{equation*}
	Note that such an $N_\epsilon$ exists thanks to~\eqref{eq:Rie_Int}. For any $N \geq N_\epsilon$, we thus find
	\begin{align*}
		0 \leq & \inf_{\bm{\delta} \in \D^+_{KN}} \sum_{k \in \K} \sum_{n \in \set{N}} f\left( \delta_{k,n}, t_k \right) \frac{\Delta t}{N} - \inf_{\delta \in \D^+} \int_{0}^{T} f\left( \delta(t),t \right) \, \mathrm{d}t \\
		\leq & \inf_{\bm{\delta} \in \D^+_{KN}} \sum_{k \in \K} \sum_{n \in \set{N}} f\left( \delta_{k,n}, t_k \right) \frac{\Delta t}{N} - \int_0^T f\left( \delta^\star(t), t \right) \, \mathrm{d}t + \epsilon \\
		\leq & \sum_{k \in \K} \sum_{n \in \set{N}} f\left( \left(L^\dagger_N \delta^\star\right)_{k,n}, t_k \right) \frac{\Delta t}{N} - \int_0^T f\left( \delta^\star(t), t \right) \, \mathrm{d}t + \epsilon \leq 2\epsilon,
	\end{align*}
	where the first inequality holds because $L_N\D^+_{KN} \subseteq \D^+$, the second inequality follows from the choice of $\delta^\star$, the third inequality exploits the identity $L^\dagger_N\D^+ = \D^+_{KN}$, and the fourth inequality holds because $N \geq N_\epsilon$. As $\epsilon > 0$ was chosen arbitrarily, Equation~\eqref{eq:Lim_Inf} follows. 
	
	In order to prove that
	\begin{equation}\label{eq:Inf_dis}
		\inf_{\delta \in \D^+} \int_{0}^{T} f\left( \delta(t), t \right) \, \mathrm{d}t 
		= \inf_{\delta \in \D^+ \cap \set{L}(\T,\{0,1\})} \int_0^T f\left( \delta(t), t \right) \, \mathrm{d}t, 
	\end{equation}
	we first observe that
	\begin{align}
		\inf_{\delta \in \D^+} \int_{0}^{T} f\left( \delta(t), t \right) \, \mathrm{d}t
		= & \lim_{N \to \infty} \inf_{\bm{\delta} \in \D^+_{KN}} \sum_{k \in \K} \sum_{n \in \set{N}} f\left( \delta_{k,n}, t_k \right) \frac{\Delta t}{N} \notag \\
		= & \lim_{N \to \infty} \inf_{\bm{\delta} \in \D^+_{KN} \cap \{0,1\}^{KN}} \sum_{k \in \K} \sum_{n \in \set{N}} f\left( \delta_{k,n}, t_k \right) \frac{\Delta t}{N}.\label{eq:Sum_bin}
	\end{align}
	Here, the first equality follows from~\eqref{eq:Lim_Inf}, and the second equality holds because $f$ is concave in its first argument, which implies that the minimum over $\delta$ is attained at a vertex of the polyhedron $\D^+_{KN}$. As all vertices of $\D^+_{KN}$ are binary by virtue of Lemma~\ref{lem:TUM} below, we can restrict $\bm{\delta}$ to $\{0,1\}^{K \times N}$ without loss of optimality. To prove~\eqref{eq:Inf_dis}, select any $\epsilon > 0$ and $N \in \mathbb{N}$ large enough such that
	\begin{equation}
		\label{eq:Min_inf}
		\left \vert \min_{\bm{\delta} \in \D^+_{KN} \cap \{0,1\}^{K \times N}} \sum_{k \in \K} \sum_{n \in \set{N}} f\left( \delta_{k,n}, t_k \right) \frac{\Delta t}{N} - \inf_{\delta \in \D^+} \int_{0}^{T} f\left( \delta(t), t \right) \, \mathrm{d}t \right \vert \leq \epsilon.
	\end{equation}
	Note that such an $N$ exists because of~\eqref{eq:Sum_bin}. Next, let $\bm{\delta}^\star$ be a minimizer of the discrete optimization problem on the left hand side of the above expression, and set $\delta^\star = L_N \bm \delta^\star$. By~\eqref{eq:Min_inf} and because $\delta^\star$ is constant on the intervals $\T^N_{k,n}$, we thus have 
	\begin{equation*}
		\left\vert \int_0^T f(\delta^\star(t),t) \, \mathrm{d}t - \inf_{\delta \in \D^+} \int_0^T f(\delta(t),t) \, \mathrm{d}t \right\vert \leq \epsilon.
	\end{equation*}
	As $L_N \D^+_{KN} \subseteq \D^+$ and $\bm \delta^\star \in \{0,1\}^{K \times N}$, we further have $\delta^\star \in \D^+ \cap \set{L}(\T,\{0,1\})$. As $\epsilon$ was chosen arbitrarily, Equation~\eqref{eq:Inf_dis} follows.
	
	If $t \in \T$ is a multiple of $1/(KN)$ for some $N \in \mathbb{N}$, then $f(\delta(t'),t')$ can be set to $0$ for all $t' \geq t$, and the above proof remains valid with obvious minor modifications. For any other $t \in \T$, the claim follows from a continuity argument. Details are omitted for brevity.
\end{proof}

\begin{lem}\label{lem:TUM}
	For any $N \in \mathbb{N}$, all vertices of the polyhedron 
	\begin{equation*}
		\D^+_{KN} = \left\{ \bm{\delta} \in \left[0,1\right]^{KN} : \sum\limits_{l= 1+\left[m - N\Gamma/\Delta t\right]^+}^{m} \delta_l \leq N\frac{\gamma}{\Delta t} ~ \forall m = 1,\ldots,KN \right\}
	\end{equation*}
	are binary vectors.
\end{lem}

\begin{proof}
	The polyhedron $\D^+_{KN}$ can be represented more concisely as $\{\bm{\delta} \in \mathbb{R}_+^{KN}: \bm{A\delta} \leq \bm{b}\}$, where
	\begin{equation*}
		\bm{A} = 
		\begin{pmatrix}
			\bm{C} \\
			\bm{I}
		\end{pmatrix} \in \mathbb{R}^{2KN \times KN}, \quad
		\bm{b} = 
		\begin{bmatrix}
			N\frac{\gamma}{\Delta t} \bm{1} \\
			\bm{1}
		\end{bmatrix} \in \mathbb{R}^{2KN}
	\end{equation*}
	and $\bm{C} \in \mathbb{R}^{KN \times KN}$ is defined through $C_{ij} = 1$ if~$i - N\Gamma/\Delta t < j \leq i$ and~$C_{ij} = 0$ otherwise. Here, $\bm{I}$ denotes the identity matrix and $\bm{1}$ the column vector of $1$s in $\mathbb{R}^{KN}$. By construction, $\bm{A}$ is a binary matrix where the $1$s appear consecutively in each row. Proposition~2.1 and Corollary~2.10 by \cite{GN99} thus imply that $\bm{A}$ is totally unimodular. As $\bm{b} \in \mathbb{Z}^{KN}$ because of Assumption~\ref{Ass:div}, all vertices of~$\D^+_{KN}$ are integral thanks to Proposition~2.2 again by \cite{GN99}. In addition, as $D^+_{KN} \subseteq [0,1]^{KN}$, the vertices of $\D^+_{KN}$ are in fact binary vectors.
\end{proof}

We are now ready to prove Proposition~\ref{Prop:yubar}.

\begin{proof}[Proof of Proposition~\ref{Prop:yubar}]
	The claim follows if we can show that
	\begin{equation}
		\label{eq:yubar_start}
		\min_{t \in \T} \min_{\delta \in \D} y(x^b, x^r, \delta, y_0, t) = \min_{k \in \K \cup \{0\}} \min_{\bm \delta \in \D_{\K}} y_k(\bm{x^b}, \bm{x^r}, \bm \delta, y_0).
	\end{equation}
	In the first part of the proof, we reformulate the continuous non-convex minimization problem $\min_{\delta \in \D} y(x^b, x^r, \delta, y_0, t)$ as a continuous linear program. To ease notation, we set $\Delta \eta = \frac{1}{\eta^-} - \eta^+ \geq 0$ and define the auxiliary functions
	\begin{equation*}
		\chi(\delta(t),t) = \max\left\{ \eta^+ x^r(t) \delta(t), \frac{1}{\eta^-}x^r(t)\delta(t) - \Delta \eta \, x^b(t) \right\}
	\end{equation*}
	and
	\begin{equation*}
		m(t) = \max\left\{ \eta^+ x^r(t), \frac{1}{\eta^-}x^r(t) - \Delta \eta \, x^b(t) \right\}
	\end{equation*}
	for all~$t \in \mathcal{T}$. The function $\chi(\delta(t),t) $ can be viewed as a {\em nonlinear decision rule} of the uncertain frequency deviation $\delta(t)$. Using these conventions, we find
	\begin{equation}
		\label{eq:yubar_cont_ldr}
		\begin{aligned}
			\min_{\delta \in \D} y\big(x^b,x^r,\delta,y_0,t\big)
			& = \min_{\delta \in \D^+} y\big(x^b,x^r,-\delta,y_0,t\big) \\
			&= y_0 + \min_{\delta \in \D^+} \int_0^{t} \eta^+ x^b(t') - \chi\left(\delta(t'),t'\right) - d(t') \, \mathrm{d}t'\\
			& = y_0 + \min_{\delta \in \D^+ \cap \set{L}(\T, \{0,1\})} \int_0^{t} \eta^+ x^b(t') - \chi\left(\delta(t'),t'\right) - d(t') \, \mathrm{d}t'\\
			& = y_0 + \min_{\delta \in \D^+ \cap \set{L}(\T, \{0,1\})} \int_0^{t} \eta^+ x^b(t') - m(t')\delta(t') - d(t') \, \mathrm{d}t' \\
			& = y_0 + \min_{\delta \in \D^+} \int_0^{t} \eta^+ x^b(t') - m(t')\delta(t') - d(t') \, \mathrm{d}t',
		\end{aligned}
	\end{equation}
	where the first equality holds because the statements $\delta \in \D$, $-\delta \in \D$ and $\vert \delta \vert \in \D^+$ are all equivalent and because $y$ is nondecreasing in $\delta$ thanks to Proposition~\ref{Prop:y}. The second equality follows from the definitions of $y$, $y^+$, $y^-$ and $\chi$, and the third equality is a direct consequence of Lemma~\ref{lem:LDR_Opt}, which applies because $-\chi$ is concave and nonincreasing in its first argument and, by virtue of Assumption~\ref{Ass:cst}, piecewise constant in its second argument. The fourth equality holds because $\chi(\delta(t'),t') = m(t')\delta(t')$ whenever $\delta(t') \in \{0,1\}$, and the last equality follows again from Lemma~\ref{lem:LDR_Opt}. Note that $m(t')\delta(t')$ constitutes a {\em linear decision rule} of $\delta(t)$.
	
	In the second part of the proof we assume that $t = k\Delta t$ for some $k \in \K \cup \{0\}$ and show that 
	\begin{equation*}
		\min_{\delta \in \D} y(x^b, x^r, \delta, y_0, t)
		= \min_{\bm \delta \in \D_{\K}} y_k(\bm{x^b}, \bm{x^r}, \bm \delta, y_0).
	\end{equation*}
	To this end, we	define $\chi_l(\delta_l) = \max\{\eta^+ x^r_l \delta_l, \frac{1}{\eta^-}x^r_l \delta_l - \Delta \eta \, x^b_l\}$ and $m_l = \max\{ \eta^+ x^r_l, \frac{x^r_l}{\eta^-} - \Delta \eta \, x^b_l \}$ for all~$l \in \mathcal{K}$. By~\eqref{eq:yubar_cont_ldr}, we thus have
	\begin{equation}
		\label{eq:yubar_res}
		\begin{aligned}
			\min_{\delta \in \D} y\big(x^b,x^r,\delta,y_0,k\Delta t\big)
			& = y_0 + \min_{\delta \in \D^+} \int_0^{k \Delta t} \eta^+ x^b(t') - m(t')\delta(t') - d(t') \, \mathrm{d}t' \\
			& = y_0 + \min_{\delta \in \D^+} \sum_{l=1}^k \int_{\T_l} \eta^+ x^b_l - m_l\delta(t') - d_l \, \mathrm{d}t' \\
			& = y_0 + \min_{\bm{\delta} \in \D^+_\K} \Delta t \sum_{l=1}^{k} \eta^+ x^b_l - m_l\delta_l - d_l \\
			& = \min_{\bm{\delta} \in \D_\K} y_k \left( \bm{x^b}, \bm{x^r}, \bm{\delta}, y_0 \right),
		\end{aligned}
	\end{equation}
	where the second equality holds because $d, x^b$ and $x^r$ are piecewise constant by virtue of Assumption~\ref{Ass:cst}, which implies that $m(t') = m_l$ for every~$t' \in \mathcal{T}_l$. The third equality then follows from Proposition~\ref{Prop:D}. The fourth equality can be proved by reversing the arguments from~\eqref{eq:yubar_cont_ldr} with obvious minor modifications. In fact, as the frequency deviation scenarios are now piecewise constant and can be encoded by finite-dimensional vectors, the proof requires no cumbersome limiting arguments as the ones developed in the proof of Lemma~\ref{lem:LDR_Opt}. We omit the details for brevity.
	
	In the third part of the proof we assume that $t \in \T_k$ for some $k \in \K$ and show that
	\begin{equation*}
		\min_{t \in \T_k} \min_{\delta \in \D} y(x^b, x^r, \delta, y_0, t) = \min_{l \in \{k-1,k\}} \min_{\bm \delta \in \D_\K} y_l(\bm{x^b}, \bm{x^r}, \bm \delta, y_0).
	\end{equation*}
	As in the proof of Proposition~\ref{Prop:ybar}, we distinguish whether or not the vehicle is driving in period~$\T_k$. Specifically, if the vehicle is driving in period~$\T_k$, then $\bar{y}^+(t) = \bar{y}^-(t) = 0$, which implies that
	\begin{align*}
		\min_{\delta \in \D} y(x^b,x^r,\delta,y_0,t) = & \min_{\delta \in \D} y(x^b,x^r,\delta,y_0,(k-1)\Delta t) - \int_{(k-1) \Delta t}^{t} d(t') \, \mathrm{d}t \\
		= & \min_{\delta \in \D} y(x^b,x^r,\delta,y_0,k\Delta t) + \int_{t}^{k \Delta t} d(t') \, \mathrm{d}t \\
		\geq & \min_{\delta \in \D} y(x^b,x^r,\delta,y_0,k\Delta t) = \min_{\bm{\delta} \in \D_\K} y_{k}(\bm{x^b}, \bm{x^r}, \bm{\delta},y_0).
	\end{align*}
	Here, the inequality holds because $d(t) \geq 0$ for all $t\in \T_k$, and the last equality follows from~\eqref{eq:yubar_res}. Otherwise, if the vehicle is parked in period $\T_k$, then $d(t) = 0$ for all $t\in \T_k$, and hence
	\begin{equation}
		\label{eq:local_sensitivity}
		\begin{aligned}
			& \phantom{=}\min_{\delta \in \D} y(x^b,x^r,\delta,y_0,t) \\
			& = y_0 + \min_{\delta \in \D^+} \sum_{l=1}^{k-1} \int_{\T_l} \eta^+x^b_l - m_l \delta(t') - d_l \, \mathrm{d}t'
			+ \int_{(k-1)\Delta t}^{t} \eta^+x^b_k - m_k \delta(t') \, \mathrm{d}t' \\
			& = y_0 + \min_{\delta \in \D^+(t)}
			\Delta t \sum_{l=1}^{k-1} \eta^+ x^b_l - d_l +
			(t - (k - 1)\Delta t) \eta^+x^b_k -  \sum_{l=1}^{k} \int_{\T_l} m_l \delta(t') \, \mathrm{d}t'\\
			& = y_0 + \Delta t \sum_{l=1}^{k-1} \eta^+ x^b_l - d_l +
			(t - (k - 1)\Delta t) \eta^+x^b_k
			- \max_{\bm{\delta} \in \D^+_\K(t)} \Delta t \sum_{l=1}^{k} m_l \delta_l ,
		\end{aligned}
	\end{equation}
	where we use the time-dependent uncertainty sets 
	\begin{equation*}
		\D^+(t) = \bigg\{\delta \in \D^+ : \delta(t') = 0 ~\forall t' \in [t, k\Delta t]\bigg\} \quad \mathrm{and} \quad \D^+_\K(t) = \bigg\{ \delta \in \D^+_\K: \delta_k \leq \frac{t - (k-1)\Delta t}{\Delta t} \bigg\}
	\end{equation*}
	to simplify the notation. The first equality in~\eqref{eq:local_sensitivity} follows from~\eqref{eq:yubar_cont_ldr} and Assumption~\ref{Ass:cst}. Note that $\delta(t')$ does not impact the objective function of the resulting minimization problem over~$\D^+$ for any $t' > t$. It is therefore optimal to set $\delta(t') = 0$ for all $t' \geq t$ and, in particular, for all $t' \in [t, k\Delta t]$. This restriction has no impact on the objective function but maximizes nature's flexibility in selecting harmful frequency deviations $\delta(t')$ for $t' \leq t$. Hence, the second equality in~\eqref{eq:local_sensitivity} follows. As $\delta$ is now integrated against a piecewise constant function, it may be	averaged over the trading intervals without changing its objective function value. The third equality in~\eqref{eq:local_sensitivity} thus holds because~$L^\dagger \D^+(t) = \D^+_\K(t)$, which can be proved similarly to Proposition~\ref{Prop:D} by noting that
	\begin{equation*}
		( \set{L}^\dagger \delta )_k
		= \frac{1}{\Delta t} \int_{(k-1)\Delta t}^{t} \delta(t') \, \mathrm{d}t' \leq \frac{t - (k-1)\Delta t}{\Delta t}
		\quad \forall \delta \in \mathcal{D}^+(t).
	\end{equation*}
	In the following we show that~\eqref{eq:local_sensitivity} is piecewise affine in~$t$. To this end, note that the optimization problem in the last line of~\eqref{eq:local_sensitivity} can be expressed more concisely as the standard form linear program
	\begin{equation}
		\label{pb:SP}
		\begin{array}{{>{\displaystyle}c>{\displaystyle}l}}
			\min_{\bm z \geq \bm 0} & \bm c^\top \bm z \\
			\subj			   & \bm A \bm z = \bm b(t),
		\end{array}
	\end{equation}
	where $\bm z^\top = (\bm \delta^\top, \bm s^\top) \in \mathbb{R}^{K} \times \mathbb{R}^{2K}$ combines the (averaged) frequency deviations in the trading intervals with a vector of slack variables. Here, the vector $\bm c \in \mathbb{R}^{3K}$ of objective function coefficients is defined through $c_l = -m_l \Delta t$ if $l \leq k$ and $c_l = 0$ otherwise. The constraints involve the matrix 
	\begin{equation*}
		\bm A = \begin{pmatrix}
			\bm C & \bm{I} & \bm{0} \\
			\bm{I}& \bm{0} & \bm{I}
		\end{pmatrix} \in \mathbb{R}^{2K \times 3K},
	\end{equation*}
	where $\bm{C} \in \mathbb{R}^{K \times K}$ is defined through $C_{ij} = 1$ if $i - \Gamma/\Delta t < j \leq i$ and $C_{ij} = 0$ otherwise, and the vector $\bm{b}(t) \in \mathbb{R}^{2K}$ is defined through $b_l(t) = \frac{t - (k-1)\Delta t}{\Delta t}$ if $l = k+K$ and $b_l = 1$ otherwise. By Lemma~\ref{lem:TUM} and Proposition~2.1 of \cite{GN99}, $\bm A$ is totally unimodular.
	
	Note that~\eqref{pb:SP} is solvable for every $t \in \T_k$ because its feasible set is non-empty and compact. Next, choose any~$t_0$ in the interior of~$\T_k$, denote by~$\bm{B}$ an optimal basis matrix for problem~\eqref{pb:SP} at $t = t_0$, and define~$\bm{z}^\star(t) = \bm B^{-1} \bm b(t) $ for all~$t \in \T_k$. In the following, we will use local sensitivity analysis of linear programming to show that~$\bm z^\star(t)$ is optimal in~\eqref{pb:SP} for all~$t \in \T_k$. As the basis~$\bm B$ remains dual feasible when~$t$ deviates from~$t_0$, it suffices to show that
	\begin{equation}
		\label{eq:local_sen}
		\bm{z}^\star(t) = \bm{z}^\star(t_0) + \frac{t-t_0}{\Delta t} \bm B^{-1}\bm{e}_{K+k}
		\geq \bm 0 \quad \forall t \in \T_k,
	\end{equation} 
	where~$\bm{e}_{K+k}$ denotes the $(K+k)$-th standard basis vector in~$\mathbb{R}^{2K}$~\citep[p.~207]{DB97}.
	To this end, note that $\bm B$ is a non-singular square matrix constructed from $2K$~columns of~$\bm A$ and is therefore also totally unimodular. Moreover, $\bm B^{-1}$ is totally unimodular because pivot operations preserve total unimodularity~\citep[Proposition~2.1]{GN99}. Hence, we have~$\bm{B}^{-1}\bm{e}_{K+k} \in \{-1,0,1\}^{2K}$. By construction, we further have $\bm b(t) \in \{0,1\}^{2K}$ for $t = k \Delta t$, which implies that $\bm{z}^\star(k\Delta t) \in \mathbb{Z}^{2K}$. Evaluating~\eqref{eq:local_sen} at $t = k \Delta t$ then yields 
	\begin{equation*}
		\bm z^\star(t_0) = \bm z^\star(k\Delta t) - \frac{k\Delta t - t_0}{\Delta t} \bm B^{-1} \bm e_{K+k},
	\end{equation*}
	which ensures that $\bm z^\star(k \Delta t) \ge \bm 0$. Indeed, if any component of the integral vector $\bm z^\star(k \Delta t)$ was strictly negative, it would have to be smaller or equal to~$-1$. As $t_0$ resides in the interior of~$\T_k$ and thus $\vert (k\Delta t - t_0)/\Delta t \vert < 1$, the corresponding component of~$\bm z^\star(t_0)$ would then also have to be strictly negative. This, however, contradicts the optimality of~$\bm z^\star(t_0)$, which implies that $\bm z^\star(t_0) \ge \bm 0$. Hence, we have $\bm z^\star(k \Delta t) \ge \bm 0$. One can use similar arguments to prove that $\bm z^\star((k-1)\Delta t) \ge \bm 0$. As $\bm z^\star(t)$ is affine in~$t$, it is indeed non-negative for all~$t \in \T_k$.
	
	The above reasoning shows that~$\bm z^\star(t)$ is optimal in~\eqref{pb:SP} and that the minimum of~\eqref{pb:SP} is affine in~$t$ on~$\T_k$. Equation~\eqref{eq:local_sensitivity} further implies that~$\min_{\delta \in \D} y(x^b,x^r,\delta,y_0,t)$ is affine in~$t$ on~$\T_k$, and thus
	\begin{equation*}\label{eq:yubar_squezze}
		\min_{t \in \T_k} \min_{\delta \in \D} y(x^b,x^r,\delta,y_0,t) =
		\min_{l\in\{k-1,k\}} \min_{\delta \in \D} y(x^b,x^r,\delta,y_0,l\Delta t) =
		\min_{l\in\{k-1,k\}} \min_{\bm \delta \in \D_\K} y_l(\bm{x^b},\bm{x^r},\bm \delta,y_0),
	\end{equation*}
	where the second equality follows from~\eqref{eq:yubar_res}. As~$k \in \K$ was chosen arbitrarily, \eqref{eq:yubar_start} follows.
\end{proof}

\begin{Prop}\label{Prop:H} The following equality holds.
	\begin{equation*}
		\max_{\delta \in \set{\hat{D}}, y_0 \in \set{\hat{Y}}_0} \varphi(y(x^b,x^r,\delta,y_0,T)) = \max_{\bm{\delta} \in \set{\hat{D}}_\K,  y_0 \in \set{\hat{Y}}_0} \varphi(y_K(\bm{x^b},\bm{x^r},\bm{\delta},y_0))
	\end{equation*}
\end{Prop}

\begin{proof}
	By introducing an auxiliary epigraphical variable~$z$, we find
	\begin{equation}
		\begin{aligned}
			\label{eq:z}
			\max_{\delta \in \set{\hat{D}},y_0 \in \set{\hat{Y}}_0} \varphi(y(x^b,x^r,\delta,y_0,T))
			& = \left\{ \begin{array}{*1{>{\displaystyle}c}*1{>{\displaystyle}l}}
				\min_{z} & z \\
				\subj & z \geq \max_{\delta \in \set{\hat{D}}, y_0 \in \set{\hat{Y}}_0} \varphi(y(x^b,x^r,\delta,y_0,T))
			\end{array} 
			\right. \\
			& = \left\{ \begin{array}{*1{>{\displaystyle}c}*2{>{\displaystyle}l}}
				\min_{z} & z \\
				\subj & z \geq \max_{\delta \in \set{\hat{D}}, y_0 \in \set{\hat{Y}}_0} 
				q_n y(x^b,x^r,\delta,y_0,T) + r_n & \forall n \in \set{N}
			\end{array} 
			\right. \\
			& = \left\{ \begin{array}{*1{>{\displaystyle}c}*2{>{\displaystyle}l}}
				\min_{z} & z \\
				\subj & z \geq \max_{\delta \in \set{\hat{D}}, y_0 \in \set{\hat{Y}}_0} 
				q_n y_K(\bm{x^b},\bm{x^r},\bm \delta,y_0) + r_n & \forall n \in \set{N}
			\end{array} 
			\right. \\
			& = \max_{\bm{\delta} \in \set{\hat{D}}_\K,  y_0 \in \set{\hat{Y}}_0} \varphi(y_K(\bm{x^b},\bm{x^r},\bm{\delta},y_0)).
		\end{aligned}
	\end{equation}
	The second equality follows from the definition of $\varphi$ and the third equality follows from Propositions~\ref{Prop:ybar} and~\ref{Prop:yubar}, which apply since $\hat{\D}$ and $\hat{\D}_\K$ have the same structures as $\D$ and $\D_{\K}$, respectively.
\end{proof}

\begin{proof}[Proof of Theorem~\ref{th:time}]
	The claim follows immediately from Propositions~\ref{Prop:yc}--\ref{Prop:H}.
\end{proof}


\begin{proof}[Proof of Theorem~\ref{th:lr}]
	By introducing embedded optimization problems that evaluate the (decision-dependent) worst-case frequency deviation scenarios and by replacing the uncertain initial state-of-charge in each robust constraint with its (decision-\emph{in}dependent) worst-case value, \eqref{pb:R} becomes
	\begin{equation}
		\label{pb:R'}
		\begin{array}{>{\displaystyle}c*3{>{\displaystyle}l}>{\displaystyle}c}
			\min_{\bm{x^b},\bm{x^r} \in \set{X}_\K, z \in \mathbb{R}} & \multicolumn{4}{>{\displaystyle}l}{c_\K(\bm{x^b},\bm{x^r}) + z} \\
			\subj & \max_{\bm{\delta} \in \D_{\K}} y^+(x^b_k,x^r_k,\delta_k) &\leq \bar{y}^+_k 
			& \forall k \in \K & \hspace{0.75cm}\text{(a)} \\
			& \max_{\bm{\delta} \in \D_{\K}} y^-(x^b_k,x^r_k,\delta_k) &\leq \bar{y}^-_k 
			& \forall k \in \K & \hspace{0.75cm}\text{(b)} \\
			& \max_{\bm \delta \in \D_{\K}}  y_{k}(\bm{x^b},\bm{x^r},\bm{\delta}, \bar y_0) &\leq \bar{y}
			& \forall k \in \K \cup \{0\} & \hspace{0.75cm}\text{(c)} \\
			& \min_{\bm \delta \in \D_{\K}}  y_{k}(\bm{x^b},\bm{x^r},\bm{\delta}, \ubar y_0) &\geq \ubar{y} 
			& \forall k \in \K \cup \{0\} & \hspace{0.75cm}\text{(d)} \\
			& \max_{y_0 \in \hat{\set{Y}}_0} \max_{\bm{\delta} \in \set{\hat{D}}_\K} \varphi\big(y_K(\bm{x^b},\bm{x^r},\bm{\delta},y_0)\big) 
			& \leq z & &\hspace{0.75cm}\text{(e)}
		\end{array}
	\end{equation}
	Here, the worst-case cost-to-go has been moved from the objective function to the constraints by introducing the auxiliary epigraphical variable~$z$. To show that~\eqref{pb:R'} is equivalent to~\eqref{pb:LR}, we reuse several results derived for the proof of Theorem~\ref{th:time}. First, by Equation~\eqref{eq:yc} in the proof of Proposition~\ref{Prop:yc} the maximum charging power in~(\ref{pb:R'}a) equals
	\begin{equation*}
		\max_{\bm{\delta} \in \D_{\K}}  y^+(x^b_k,x^r_k,\delta_k) = x^r_k + x^b_k.
	\end{equation*}
	Using similar arguments, it can be shown that the maximum discharging power in~(\ref{pb:R'}b) reduces to
	\begin{equation*}
		\max_{\bm{\delta} \in \D_{\K}}  y^-(x^b_k,x^r_k,\delta_k) = x^r_k - x^b_k.
	\end{equation*}
	Next, Equation~\eqref{eq:ybar_full} in the proof of Proposition~\ref{Prop:ybar} reveals that, for any~$k \in \K \cup \{0\}$, the maximum state-of-charge in~(\ref{pb:R'}c) is given by 
	\begin{equation*}
		\max_{\bm{\delta} \in \D_\K} y_{k}\left(\bm{x^b},\bm{x^r},\bm{\delta}, \bar y_0\right)
		= \bar y_0 + \max_{\bm{\delta} \in \D^+_\K} \Delta t \sum_{l=1}^{k} \eta^+ \left( x^b_l + \delta_l x^r_l \right) - d_l.
	\end{equation*}
	Similarly, Equation~\eqref{eq:yubar_res} in the proof of Proposition~\ref{Prop:yubar} implies that, for every~$k \in \K \cup \{0\}$, the minimum state-of-charge in~(\ref{pb:R'}d) amounts to
	\begin{equation*}
		\min_{\bm{\delta} \in \D_\K} y_k \left( \bm{x^b}, \bm{x^r}, \bm{\delta}, \ubar y_0 \right) = \ubar y_0 + \min_{\bm{\delta} \in \D^+_\K} \Delta t \sum_{l=1}^{k} \eta^+ x^b_l - m_l\delta_l - d_l,
	\end{equation*}
	where $m_l = \max\{ \eta^+ x^r_l, \frac{1}{\eta^-} x^r_l - \Delta \eta \, x^b_l \}$ constitutes an implicit function of the market decisions~$x^b_l$ and~$x^r_l$. As~(\ref{pb:R'}d) imposes a~\emph{lower} bound on the minimum state-of-charge, $m_l$ may be reinterpreted as an auxiliary epigraphical variable that satisfies~$m_l \geq \eta^+ x^r_l$ and~$m_l \geq \frac{1}{\eta^-} x^r_l - \Delta \eta \, x^b_l$.
	
	Finally, the maximum cost-to-go in~(\ref{pb:R'}e) can be reformulated as
	\begin{align*}
	    & \max_{y_0 \in \hat{\set{Y}}_0} \max_{\bm{\delta} \in \set{\hat{D}}_\K} \varphi\big(y_K(\bm{x^b},\bm{x^r},\bm{\delta},y_0)\big)
	     = \max_{y_0 \in \hat{\set{Y}}_0}\max_{\bm{\delta} \in \set{\hat{D}}_\K} \max_{n \in \set{N}} \, q_n y_K(\bm{x^b},\bm{x^r},\bm{\delta},y_0) + r_n \\
	     & \quad  =  \max_{\bm \delta \in \set{\hat{D}}_\K} \max\left\{ \max_{n \in \set{N}_+} r_n + q_n y_K(\bm{x^b},\bm{x^r},\bm{\delta},\hat{\bar y}_0), \,
	     \max_{n \in \set{N}_-} r_n + q_n y_K(\bm{x^b},\bm{x^r},\bm{\delta},\ubar{\hat y}_0),\max_{n\in\N_0} r_n
	     \right\},
	\end{align*}
	where the first equality follows from the definition of the convex piecewise affine function~$\varphi$. The second equality holds because the order of maximization is immaterial, as~$q_n > 0$ for all $n \in \set{N}_+$, $q_n < 0$ for all~$n \in \set{N}_-$ and~$q_n=0$ for all~$n\in\N_0$ and because the state-of-charge $y_K(\bm{x^b},\bm{x^r},\bm{\delta},y_0)$ increases with~$y_0$. Requiring the last expression to be smaller than or equal to~$z$ is equivalent to
	\begin{align*}
	    \max_{y_0 \in \hat{\set{Y}}_0} \max_{\bm{\delta} \in \set{\hat{D}}_\K} \varphi\big(y_K(\bm{x^b},\bm{x^r},\bm{\delta},y_0)\big) \leq z 
	    \iff 
	    \begin{cases}
	        \displaystyle \max_{\bm{\delta} \in \set{\hat{D}}_\K} y_K(\bm{x^b},\bm{x^r},\bm{\delta},\hat{\bar y}_0) \leq (z-r_n)/q_n & \forall n \in \set{N}_+ \\
	        \displaystyle \min_{\bm{\delta} \in \set{\hat{D}}_\K} y_K(\bm{x^b},\bm{x^r},\bm{\delta},\ubar{\hat y}_0)  \geq (z-r_n)/q_n & \forall n \in \set{N}_- \\
	        r_n\leq z & \forall n\in\N_0.
	    \end{cases}
	\end{align*}
	As~$\D_{\K}$ and~$\hat \D_{\K}$ have the same structure, the embedded optimization problems over $\bm \delta \in \hat{\set{D}}_\set{K}$ admit similar linear reformulations as the embedded optimization problems over $\bm \delta \in \set{D}_\set{K}$ in~(\ref{pb:R'}c) and~(\ref{pb:R'}d). Substituting all obtained reformulations into~\eqref{pb:R'} yields~\eqref{pb:LR}.
\end{proof}

\begin{proof}[Proof of Theorem~\ref{th:lp}]
	Problem~\eqref{pb:LR} can be reformulated as a linear program by using the standard machinery of robust optimization~\citep{DB04,AB09}. For example, the robust upper bound on the state-of-charge for a fixed~$k \in \K \cup \{0\}$ is equivalent to
	\begin{equation}
		\label{eq:max_primal}
		\bar y_0 + \Delta t \sum_{l = 1}^k \eta^+ \left( x^b_l + \delta_l x^r_l \right) - d_l \leq \bar y \quad \forall \bm \delta \in \D^+_\K
		\iff
		\max_{\bm \delta \in \D^+_\K} \Delta t \sum_{l = 1}^k \eta^+ \left( x^b_l + \delta_l x^r_l \right) - d_l \leq \bar y - \bar y_0.
	\end{equation}
	By strong linear programming duality, the maximization problem in~\eqref{eq:max_primal} is equivalent to
	\begin{equation}
		\label{eq:min_dual}
		\begin{array}{*1{>{\displaystyle}c}*2{>{\displaystyle}l}}
			\min_{\bm \Lambda^+, \bm \Theta^+ \in \mathbb{R}^{K \times K}_+} & \multicolumn{2}{>{\displaystyle}l}{\sum_{l = 1}^k \Delta t \left( \eta^+ x^b_l + \Lambda^+_{k,l} - d_l \right) + \gamma \Theta^+_{k,l} } \\
			\subj &
			\Lambda^+_{k,l} + \sum\limits_{i = l}^{j(k,l)} \Theta^+_{k,i} \geq \eta^+ x^r_l & \forall l \in \K: \; l \leq k, \\
		\end{array}
	\end{equation}
	and the minimum of~\eqref{eq:min_dual} is smaller or equal to~$\bar y - \bar y_0$ if and only if problem~\eqref{eq:min_dual} has a feasible solution whose objective value is smaller or equal to~$\bar y - \bar y_0$. Therefore, the robust constraint~\eqref{eq:max_primal} is equivalent to the following system of ordinary linear constraints. 
	\begin{equation*}
		\label{eq:min_dual_const}
		\begin{array}{*3{>{\displaystyle}l}}
			\sum_{l = 1}^k \Delta t \left( \eta^+ x^b_l + \Lambda^+_{k,l} - d_l \right) + \gamma \Theta^+_{k,l} 
			& \leq \bar y - \bar y_0 \\
			\Lambda^+_{k,l} + \sum\limits_{i = l}^{I(k,l)} \Theta^+_{k,i} \geq \eta^+ x^r_l \quad \forall l \in \K: \; l \leq k \\
		\end{array}
	\end{equation*}
	The remaining robust constraints in~\eqref{pb:LR} can be simplified in a similar manner.
\end{proof}

\begin{proof}[Proof of Theorem~\ref{th:c}]
    Since $x^\uparrow_v(t)$ and $x^\downarrow_v(t)$ are nonnegative for all $v \in \V$ and all $t \in \T$, the functions $y^+$ and $y$ remain nondecreasing in $\delta$ and $y^-$ remains nonincreasing in $\delta$, as in Proposition~\ref{Prop:y}. Thanks to the symmetry of~$\D$, the upper bounds on the charging rate and the state-of-charge of each vehicle will thus hold for all $\delta \in \D$ if and only if they hold for all $\delta \in \D_+$. Similarly, the upper bound on the discharing rate and the lower bound on the state-of-charge will hold if and only if they hold for all $-\delta \in \D_+$. For a fixed time~$t$ and vehicle~$v$, we then have
    \begin{align*}
        \max_{\delta \in \D} y^+(x^b_v(t), x^\uparrow_v(t), x^\downarrow_v(t), \delta(t))
        = & \max_{\delta \in \D^+} y^+(x^b_v(t), x^\uparrow_v(t), x^\downarrow_v(t), \delta(t)) \\
        = & \max_{\delta \in \D^+} \left[ x^b_v(t) + \left[ \delta(t) \right]^+ x^{\downarrow}_v (t) - \left[ \delta(t) \right]^- x^{\uparrow}_v(t) \right]^+ \\
        = & \max_{\delta \in \D^+} \left[ x^b_v(t) + \delta(t) x^{\downarrow}_v (t) \right]^+
        = \max_{\delta \in \D} y^+(x^b_v(t), 0, x^\downarrow_v(t), \delta(t)).
    \end{align*}
    The maximum charging rate depends thus on the frequency regulation bids only through the bid for down-regulation. To prove the upper bound on the charging rate, we can thus set $x^r(t) = x_v^\downarrow(t)$ and apply assertion~($i$) of Proposition~\ref{Prop:yc}. The reasoning for the other robust bounds is similar and omitted for brevity. In summary, problem~\eqref{pb:C} is thus equivalent to the linear program
    \begin{equation*}
    \begin{array}{*1{>{\displaystyle}c}*2{>{\displaystyle}l}}
    \min & \multicolumn{2}{>{\displaystyle}l}{c_\K\left(\sum_{v \in \V} \bm{x^b}_v, \sum_{v \in \V} \bm x^\uparrow_v \right) + \sum_{v \in \V} z_v} \\ 
    \subj
    & \bm{x^b}_v, \bm x^\uparrow_v, \bm x^\downarrow_v \in \set{X}_\K,~z_v\in\R
    & \forall v \in \V \\
    & \bm m_v, \bm \lambda^+_v, \bm \lambda^-_v, \bm \theta^+_v, \bm \theta^-_v \in \R^K_+
    & \forall v \in V \\
    & \bm \Lambda^+_v, \bm \Lambda^-_v, \bm \Theta^+_v, \bm \Theta^-_v \in \mathbb{R}^{K\times K}_+
    & \forall v \in V \\
    & \multicolumn{1}{>{\displaystyle}l}{
    x^\downarrow_{v,k} + x^b_{v,k} \leq \bar{y}^+_{v,k},
    \quad
    x^\uparrow_{v,k} - x^b_{v,k}
    \leq \bar{y}^-_{v,k}} & \forall v \in \V, \forall k \in \K \\
    & \multicolumn{1}{>{\displaystyle}l}{
    m_{v,k} \geq \eta^+_v x^\uparrow_{v,k}, 
    \quad
    m_{v,k} \geq \frac{1}{\eta^-_v} x^\uparrow_{v,k} - \Delta \eta_v \, x^b_{v,k} }
    & \forall v \in \V, \forall k \in \K \\	
    & \sum_{l=1}^k \Delta t \left( \eta^+_v x^b_{v,l} + \Lambda^+_{v,k,l} - d_{v,l} \right)  + \gamma \Theta^+_{v,k,l} \leq \bar{y}_v - \bar y_{0,v} & \forall v \in \V, \forall k \in \K \cup\{0\}\\
    &\sum_{l = 1}^k \Delta t \left( \eta^+_v x^b_{v,l} - \Lambda^-_{v,k,l} - d_{v,l} \right) - \gamma \Theta^-_{v,k,l} \geq \ubar{y}_v - \ubar y_{0,v} & \forall v \in \V, \forall k \in \K \cup\{0\}\\
    & \sum_{k \in \K} \Delta t \left( \eta^+_v x^b_{v,k} + \lambda^+_{v,k} - d_{v,k} \right) + \hat \gamma \theta^+_{v,k} \leq \frac{z_v-r_{v,n}}{q_{v,n}} - \hat{\bar y}_{0,v} & \forall v \in \V, \forall n \in \mathcal{N}_+ \\
    & \sum_{k \in \K} \Delta t \left( \eta^+_v x^b_{v,k} - \lambda^-_{v,k} - d_{v,k} \right) - \hat \gamma \theta^-_{v,k} \geq \frac{z_v-r_{v,n}}{q_{v,n}} - \hat{\ubar y}_{0,v} & \forall v \in \V, \forall n \in \mathcal{N}_- \\
    & r_{v,n} \leq z_v &  \forall v \in \V,~\forall n\in\N_0\\
    & \Lambda^+_{v,k,l} + \sum\limits_{i = l}^{I(k,l)} \Theta^+_{v,k,i} \geq \eta^+_v x^\downarrow_{v,l}, \quad
    \Lambda^-_{v,k,l} + \sum\limits_{i = l}^{I(k,l)} \Theta^-_{v,k,i} \geq m_{v,l} & \forall v \in \V, \forall k,l \in \K:\; l \leq k \\
    & \lambda^+_{v,k} + \sum\limits_{i = k}^{\hat I(K,k)} \theta^+_{v,i} \geq \eta^+_v x^\downarrow_{v,k}, \quad
    \lambda^-_{v,k} + \sum\limits_{i = k}^{\hat I(K,k)} \theta^-_{v,i} \geq m_{v,k} & \forall v \in \V, \forall k \in \K \\
    & \sum_{v \in \V} x^\uparrow_{v,k} - x^\downarrow_{v,k} = 0 & \forall k \in \K.
    \end{array}
    \end{equation*}
    Different vehicles are only coupled through the last equality constraint, which ensures that the aggregate participation in the frequency regulation market is symmetric in each period $k\in\mathcal K$.
\end{proof}

\end{document}